\documentclass[titlepage,10pt]{amsart}
\usepackage{xcolor}
\usepackage[shortlabels]{enumitem}
\usepackage{mathtools} 
\usepackage{amsaddr}
\mathtoolsset{showonlyrefs}

\usepackage{tikz-cd}

\def\R{\mathbb{R}}
\def\N{\mathbb{N}}
\def\Z{\mathbb{Z}}

\def\grad{\operatorname{grad}}
\def\sgn{\operatorname{sgn}}
\def\rank{\operatorname{rank}}
\def\codim{\operatorname{codim}}
\newcommand{\spower}[2]{\sgn\left(#1\right) \left|#1\right|^{#2-1}}

\newcommand{\sn}{\operatorname{sn}}
\newcommand{\SN}{\operatorname{SN}}

\title[]{The eigenvalues and eigenfunctions of the non-linear equation associated to second order Sobolev embeddings}
\author{Lyonell Boulton} 
\address{Department of Mathematics and Maxwell Institute for Mathematical Sciences, Heriot-Watt University, Edinburgh, EH14 4AS, UK.}
\email{L.Boulton@hw.ac.uk}

\author{Jan Lang}
\address{Department of Mathematics, The Ohio State University, 231 West 18th Avenue, Columbus, OH 43210-1174, USA.}
\email{Lang@math.osu.edu}

\date{7th May 2023}

\newtheorem{Lemma}{Lemma}[section]
\newtheorem{Corollary}{Corollary}[section]
\newtheorem{Theorem}{Theorem}[section]
\newtheorem{Remark}{Remark}[section]
\newtheorem{Definition}{Definition}[section]

\begin{document}

\begin{abstract}
We consider the non-linear eigenvalue equations characterizing $L^p$ into $L^q$ Sobolev embeddings of second order for Navier boundary conditions at both ends of a line segment. We give a complete description of the s-numbers and the extremal functions in the general case $(p,q)\in(1,\infty)^2$. Among other results, we show that these can be expressed in terms of those of related first order embeddings, if and only if $\frac{1}{p}+\frac{1}{q}=1$.   Our findings shed new light on the surprising nature of higher order Sobolev spaces in the Banach space setting. 

\vspace{2cm}

\tableofcontents

\end{abstract}

\maketitle

\section{Introduction} \label{section1}

Let us begin by recalling the classical theory of first order Sobolev embeddings in 1D.  For $t_0>0$, let the line segment $\mathcal{I}=[0,t_0]$. For $1<p<\infty$, let $W^{1,p}_{0}\equiv W^{1,p}_{0}(\mathcal{I})$ be the closure of $C^{\infty}_0(\operatorname{Int}\mathcal{I})$ in the real Sobolev space $W^{1,p}\equiv W^{1,p}(\mathcal{I})$ with respect to the norm
\[\|u\|_{W^{1,p}}= \|u\|_{L^p}+\|u'\|_{L^p}.   \] 
The seminorm $\|u\|_{W^{1,p}_{0}}=\|u'\|_{L^p}$ is also a norm equivalent to $\|u\|_{W^{1,p}}$ in $W^{1,p}_{0}$ and from now on we endow this space with it. Both $W^{1,p}$ and $W^{1,p}_0$ are embedded into $L^q\equiv L^q(\mathcal{I})$ for all $1<q<\infty$. Consider the identity maps
\[E_1: W^{1,p}_{0} \hookrightarrow L^q ,\]
which are compact operators. As the underlying spaces are reflexive and strictly convex, then there exist non-zero functions $u_{\mathrm{D}}\in W^{1,p}_{0}$ realizing the operator norm of $E_1$. That is, such that
\[\|E_1\|=\sup_{u\in W^{1,p}_{0}} {\|u\|_{L^q} \over  \|u'\|_{L^p}} = {\|u_{\mathrm{D}}\|_{L^q} \over  \|u'_{\mathrm{D}}\|_{L^p}}>0.
\]
These optimal functions can be fully described in terms of an eigenvalue equation which derives from a duality map formulation \cite[Sect.~3.2]{EL-book1}, as follows. 

We seek for $u_{\mathrm{D}}$, such that $\|u_{\mathrm{D}}\|_{L^q}=\|E_1\|$ and $\|u_{\mathrm{D}}\|_{W^{1,p}_{0}}=1$. See \cite[Propositions~1.10 and 1.11]{EL-book1}. Formally, this corresponds to a solution of the ``infinite-dimensional Lagrange multipliers'' equation
\begin{equation}  \label{prop11.1}
    \lambda \left\langle v, \grad \|u\|_{L^q} \right\rangle_{L^{q'}}=
 \left\langle v, \grad \|u\|_{W^{1,p}_{0}} \right\rangle_{L^{p'}}
\end{equation}
for all $v\in C^{\infty}_0(\operatorname{Int}\mathcal{I})$,
where the eigenvalue $\lambda=\lambda_1=\|E_1\|^{-q}$ is minimal. 
We write $p'=\frac{p}{p-1}$, so that the usual identity $\frac{1}{p}+\frac{1}{p'}=1$ holds.
In the sense of distributions, the G\^{a}teux derivatives of the norms are
\[ 
\grad \|u\|_{L^q}=\|u\|_{L^q}^{1-q} \spower{u(t)}{q} \qquad \text{for }u\in L^q\setminus\{0\}
\] 
and 
\[
   \grad \|u\|_{W_0^{1,p}}=-\|u'\|_{L^p}^{1-p} \left(\spower{u'(t)}{p}\right)' \qquad \text{for } u\in W^{1,p}_{0}\setminus\{0\},
\]
respectively.
Then, \eqref{prop11.1} with $u_{\mathrm{D}}$ normalized as above and the representation of $\lambda_1$, yield
\[
    \lambda_1\!\!\int_0^{t_0} \!\!\!\!\! v(t) \spower{u_{\mathrm{D}}(t)}{q} \!\mathrm{d}t\!=\!\!\!
     \int_0^{t_0}\!\!\!\!\! v'(t) \spower{u_{\mathrm{D}}'(t)}{p}\!\mathrm{d}t,
\]
for all $v\in C^{\infty}_0(\operatorname{Int}\mathcal{I})$.

By following the arguments described in \cite[Sect.~1.3]{EL-book1} (or in \cite{EEH01} and \cite{EEH02}),  we then see that the latter identity renders the eigenvalue equation
\begin{equation} \label{classical_embedding}
\begin{aligned} 
-(\spower{u'}{p})'&=  \lambda \spower{u}{q} \\
u(0)=u(t_0)&=0,
\end{aligned} 
\end{equation}
of which $u_{\mathrm{D}}$ is an eigenfunction associated to the smallest eigenvalue $\lambda=\lambda_1>0$. The other extremal functions of the Sobolev embedding $E_1$ are related to the other eigenfunctions of \eqref{classical_embedding}. For $p\not=q$, this equation is not homogeneous, so extra ansatz should be imposed. We will elaborate fully on this, rigorously, in the context of second order Sobolev embeddings, in sections~\ref{secder} and \ref{section6}. 

The eigenpairs of \eqref{classical_embedding} have been studied significantly in the past. See Section~\ref{section5} or the list of references in \cite{DM,EL-book1}, for details on this. The eigenfunctions can be described explicitly in terms of generalized trigonometric functions related to inverse incomplete Beta functions. The norm of $E_1$ was first computed in \cite{Ta1} in terms of Beta functions. It is also known that the  different notions of s-numbers of $E_1$ coincide. Moreover, they scale as $\frac{1}{n}$ in their index $n$, \cite{EL2012}. 

From the above well settled theory, questions such as the following naturally arise. 
\begin{itemize}
\item Is there an analogue eigenvalue problem for Sobolev embeddings of higher order? 
\item If so, can we expect a full characterization of the extremal functions in terms of known special functions?
\item Once the equation is posed, what properties do the eigenpairs have, compared to those of \eqref{classical_embedding}? 
\item Is there any obvious relation between the first and the higher order optimizers?
\end{itemize} 
In the present paper we address these questions for the specific case of the second order Sobolev space \[W^{2,p}_{\mathrm{D}}\equiv W^{2,p}_{\mathrm{D}}(\mathcal{I})=\{u\in W^{2,p}(\mathcal{I})\,:\, u(0)=u(t_0)=0\} \]
 equipped with the norm
\[\|u\|_{W^{2,p}_{\mathrm{D}}}= \|u''\|_{L^p}
\] 
and the corresponding compact embedding 
\[E_2: W^{2,p}_{\mathrm{D}} \hookrightarrow L^q.\] 
In Section~\ref{secder} we derive the associated eigenvalue problem which serves as analogue to \eqref{classical_embedding}. See \eqref{original}. Formally, this corresponds to replacing first order differentiation with second order differentiation and an extra pair of boundary conditions, but this has to be justified rigorously. Periodic solutions of this equation will be the central object of study in sections~\ref{section2}--\ref{section4}, where we show existence and uniqueness in a suitable sense. One of the crucial steps for the proof of the latter relies on earlier fundamental work of Benedikt, \cite{Be2004}. To the best of our knowledge, only the homogeneous case $p=q$ seems to have been studied in some detail. Concretely, Dr{\'a}bek and {\^ O}tani reported in \cite{DO}, that the first eigenfunction exists, that it is symmetric with respect to $t_0/2$ and that the higher eigenfunctions are generated from this first eigenfunction by re-scaling. Below, we extend these and other related statements to all $(p,q)\in(1,\infty)^2$.

In Section~\ref{section6} we determine explicit bounds for the s-numbers of $E_2$, after examining the relation with the first order embedding $E_1$ in Section~\ref{section5}. We give close formulas for the s-numbers of $E_2$ whenever $q=p'$ and argue that this is the only case in which these are given in terms of simple expressions involving Beta functions.


\section{Derivation of the equation and main results} \label{secder}
We now consider the rigorous derivation of the non-linear eigenvalue equation for $E_2$ via duality maps. Basic terminology and background results can be found in \cite{EEH01,EEH02} and \cite[Chapter 2]{EE-book}.

Let $0\not = T:X \to Y$ be a general compact map, where $X,\, Y,\, X^*$ and $Y^*$, are real strictly convex Banach spaces. Let $\tilde{J}_X$ and $\tilde{J}_Y$ be given by $ \tilde{J}_X(x)= \grad \|x\|_X$ and  $ \tilde{J}_Y(y)= \grad \|y\|_Y$, where ``$\grad$'' is derivative in the G\^ateaux sense. Then, \cite[Proposition~1.10]{EL-book1}, there exists $x_1\in X$ such that $\|x_1\|_X=1$ and $\|T(x_1)\|_{Y}=\|T\|$. Moreover, \cite[Proposition 1.11]{EL-book1}, in the following diagram, 
	\begin{equation}
		\begin{tikzcd}[row sep=scriptsize]
			 X \dar["\tilde{J}_X"]\rar["T"] & Y \dar["\tilde{J}_Y"]\\
			 X^*  & \lar["T^*"] Y^* 
		\end{tikzcd}
	\end{equation}
$x=x_1$ satisfies the equation
\begin{equation}
T^*\tilde{J}_Y Tx= \nu \tilde{J}_X x \label{j-equation}
\end{equation} 
for $\nu=\|T\|$.
The case of first order Sobolev embeddings described above, corresponds to taking  $X=W^{1,p}_{0}$, $Y=L^q$, $T=E_1$ and \eqref{j-equation} unpacks into \eqref{prop11.1}.

For second order embeddings, we seek for non-zero extremal elements $u_\mathrm{D}\in W_{\mathrm{D}}^{2,p}$ such that
\[\|E_2\|=\sup_{u\in W_{\mathrm{D}}^{2,p}} {\|u\|_{L^q} \over  \|u''\|_{L^p}} = {\|u_{\mathrm{D}}\|_{L^q} \over  \|u''_{\mathrm{D}}\|_{L^p}}.
\]
In order to characterize $u_{\mathrm{D}}$, set $X=W^{2,p}_{\mathrm{D}}$, $Y=L^q$ and $T=E_2$. Both $X$ and $Y$ are reflexive and strictly convex. Also  $Y^*$, being $L^{q'}$, is strictly convex. Furthermore, $\|.\|_X$ is G\^ateaux-differentiable on $X\setminus \{0\}$, hence  $X^*$ is also strictly convex, \cite[Proposition 1.8]{EL-book1}. In the sense of distributions,
\[\tilde{J}_Y(u)= \|u\|_{L^q}^{1-q} \spower{u}{q}, \qquad u\in L^q\setminus \{0\},
\]   
and
\[\tilde{J}_X(u)= \|u''\|_{L^p}^{1-p}  \left(\spower{u''}{p}\right)'', \qquad u\in W^{2,p}_{\mathrm{D}}\setminus \{0\}.
\]
The equation \eqref{j-equation} takes the form
\begin{equation} \begin{aligned}
 \|u\|_{L^q}^{1-q}\!\! &\int_0^{t_0}\!\!\!\! v \spower{u}{q}\!\! \mathrm{d}x\! \\ &=\!\! \|E_2\|
\|u''\|_{L^p}^{1-p}\!\! \int_0^{t_0}\!\!\!\! v'' \spower{u''}{p} \!\!\mathrm{d}x,
\end{aligned}
\label{j-equation 2} \end{equation}  
 for all $v\in C^\infty_0(\operatorname{Int}\mathcal{I})$.
By compactness, there exists a (weak) solution $u=u_{\mathrm{D}}\in W_{\mathrm{D}}^{2,p}$, satisfying $\|u_{\mathrm{D}}''\|_{L^p}=1$ and $\|u_{\mathrm{D}}\|_{L^q}=\|E_2\|$. We now derive the non-linear eigenvalue equation representing all turning points of \eqref{j-equation 2}.

Firstly, recall that the Poisson equation with Dirichlet boundary conditions is uniquely solvable in $L^{q'}$. That is, for
 each $f\in L^{q'}$ there exists $g \in W^{2,q'}_{\mathrm{D}}$ such that $-{g}''=f$.
In particular, there exists $\tilde{u} \in W^{2,q'}_{\mathrm{D}}$ such that $-\tilde{u}''=\spower{u_{\mathrm{D}}}{q} \in L^{q'}$. In weak form, 
\begin{equation*}
\int_0^{t_0} v \spower{u_{\mathrm{D}}}{q} = -\langle v, \tilde{u}''\rangle_{L^{q'}} 
		= -\int_0^{t_0}v'' \tilde{u} 
\end{equation*}
for all $ v\in  C^\infty_0(\operatorname{Int}\mathcal{I})$. Substituting into \eqref{j-equation 2}, we then have that
\[ \int_0^{t_0}v''\left(\|E_2\|^q \spower{u_{\mathrm{D}}''}{p} +\tilde{u} \right)=0
\]
for all $v\in C^\infty_0(\operatorname{Int}\mathcal{I})$.
Hence 
\[\tilde{u}=-\|E_2\|^q \spower{u_{\mathrm{D}}''}{p}\in L^{p'}\cap W^{2,q'}_{\mathrm{D}}
\] 
and  moreover
 \[
 \|E_2\|^q \left(\spower{u_{\mathrm{D}}''}{p} \right)''=\spower{u_{\mathrm{D}}}{q}.\]
Thus, $u_{\mathrm{D}}$ is an eigenfunction of the differential equation
 \begin{equation*} \label{theoriginalODE} \nu \left(\spower{u''}{p} \right)''=\spower{u}{q}
 \end{equation*}
with boundary conditions
\begin{equation*}\label{theoriginalBC}u(0)=u(t_0)=u''(0)=u''(t_0)=0.\end{equation*}
Here, the first two boundary conditions are consequence of the inclusion
\[
u_{\mathrm{D}}\in W^{2,p}_{\mathrm{D}}
\]
and the last two are consequence of the inclusion
\[
     |u_{\mathrm{D}}''|^{p-1}\in W^{2,q'}_{\mathrm{D}}.
\]

We therefore obtain a characterization of the norm of the embedding $E_2$ in terms of an eigenvalue equation, which can be re-cast as,  \begin{equation} \label{original}
\begin{aligned}
&(\spower{u''}{p})''=\lambda \spower{u}{q} & \qquad \qquad 0\leq t \leq t_0 \\
&u(0)=u(t_0)=u''(0)=u''(t_0)=0,
\end{aligned}
\end{equation}
for eigenpairs $u\not=0$ and $\lambda>0$. One of our main goals below will be to prove existence and uniqueness (in a suitable sense) of the eigenfunctions. The next theorem, about those which are positive, is a  template for this and it represents the first main contribution of this paper. Subject to scaling, this positive eigenfunction will turn out to be the extremal functions $u_{\mathrm{D}}$ for $\|E_2\|$. The branching of the statement  is a consequence of the fact that the equation is homogeneous if and only if $p=q$.

\begin{Theorem} \label{existence}
Let $t_0>0$ be fixed. If $p\not=q$, then for all $\lambda>0$ there exists a unique eigenfunction of \eqref{original} positive on $(0,t_0)$. If $p=q$, then there exists a unique $\lambda>0$ (depending on $t_0$) such that an eigenfunction of \eqref{original} positive on $(0,t_0)$ exists and this eigenfunction is unique up to multiplication by a scalar.
\end{Theorem}

The characterization of the strict s-numbers of $E_2$ by means of \eqref{original} follows arguments similar to those given above. In Section~\ref{section6} we will determine bounds for these s-numbers through a suitable normalization and scaling of the eigenpairs. At this point we highlight that, despite of the clear formal connections between \eqref{original} and its first order counterpart \eqref{classical_embedding},  the periodic solutions of \eqref{original} cannot be expressed in simple terms by means of those of \eqref{classical_embedding} for general $(p,q)\in(1,\infty)^2$.  This point will be addressed in Section~\ref{section5}.


\section{System formulation and background properties} \label{section2}
If  $(u,\lambda)$ is an eigenpair of \eqref{original}, integrating by parts twice and applying the boundary conditions, gives
\begin{equation} \label{eigenvalue and ratio}
       \lambda= \frac{\|u''\|^p_{L^p}}{\|u\|^q_{L^q}}\geq 0.
\end{equation}
The boundary conditions prevent $u$ from being a linear function, other than $u= 0$. Hence, $\lambda>0$. Without further mention, from now on we will assume that this is the case.

Let
\begin{equation} \label{subs}
\begin{gathered}
     u_1(t)=u(t), \quad u_2(t)=u'(t),  \\
      w_1(t)=-\spower{u''(t)}{p} \quad \text{and} \quad
      w_2(t)=-(\spower{u''(t)}{p})'
\end{gathered}
 \end{equation}
 where $u(t)$ is a solution of \eqref{original}.
Then,
\begin{equation} \label{differential} \tag{$*_\mathrm{D}$}
\begin{aligned} 
u_1'(t)&=u_2(t) \\
w_1'(t)&=w_2(t)
\end{aligned} \qquad
\begin{aligned}
u_2'(t)&=-\spower{w_1(t)}{p'} \\
w_2'(t)&=-\lambda \spower{u_1(t)}{q}.
\end{aligned}
\end{equation}
Integrating each component, gives
\begin{equation} \label{integral} \tag{$*_\mathrm{I}$}
\begin{aligned}
u_1(t)&=\int_0^t u_2(s) \mathrm{d}s \\
w_1(t)&=\int_0^t w_2(s)\mathrm{d}s   
\end{aligned}
\qquad
\begin{aligned}
u_2(t)&=\alpha-\int_0^t \spower{w_1(s)}{p'}\mathrm{d}s \\
w_2(t)&=\beta-\lambda \int_0^t \spower{u_1(s)}{q} \mathrm{d}s
\end{aligned}
\end{equation}
for $\alpha=u_2(0)$ and $\beta=w_2(0)$. We can also write \eqref{differential} as
\begin{equation}   \label{vector} \tag{$*_\mathrm{V}$}
     \underline{\varphi}'(t)=F(\lambda;\underline{\varphi}(t))
\end{equation}
where
\[
     \underline{\varphi}(t)=\begin{bmatrix} u_1(t)\\ u_2(t)\\ w_1(t) \\ w_2(t) \end{bmatrix}
\qquad 
\text{and} \qquad
     F(\lambda;x,y,z,w)=\begin{bmatrix}
y \\ -\spower{z}{p'} \\ w \\ -\lambda \spower{x}{q} \end{bmatrix}.
\]
When referring to the systems $(*)$ below, we mean collectively the equivalent formulations \eqref{differential}, \eqref{integral} and \eqref{vector}, with suitable initial conditions for \eqref{differential} and \eqref{vector}, depending on the context.

The boundary conditions of \eqref{original} turn into
\[
    u_1(0)=w_1(0)=u_1(t_0)=w_1(t_0)=0.
\]
Regarding the eigenvalue $\lambda$ as a fixed parameter, we therefore seek for a $C^1$ solution $\underline{\varphi}(t)$ satisfying the initial condition,
\begin{equation} \label{initialstate}
  \underline{\varphi}(0)=\underline{\zeta}= \begin{bmatrix} 0 \\ \alpha \\ 0 \\ \beta \end{bmatrix}.
\end{equation}
By means of a change of variables to $t-2t_0$, it is readily seen that, if such a solution returns to the initial position at $t=2t_0$, i.e. $\underline{\varphi}(2t_0)=\underline{\zeta}$, then we can continue it into a $2t_0$-periodic global solution. Below, we will see that this periodic solution corresponds exactly to an eigenfunction, as it is symmetric with respect to half the period with
$
  \underline{\varphi}(t_0)= \pm \underline{\zeta}.
$
Its existence is supported on suitable combinations of non-zero $\alpha$ and $\beta$. We will also see that it does not bifurcate. Therefore, it is unique, modulo the re-scaling introduced next.

\begin{Remark} \label{Remark3}
Assume that $(\lambda,u)$ is a solution of \eqref{original} for $t_0>0$ and that $a,b>0$. Then, $(\tilde{\lambda},\tilde{u})$ 
where 
\begin{equation} \label{subst (1)}
     \tilde{u}(t)=au(bt), \quad \tilde{\lambda}=\lambda a^{p-q}b^{2p} \quad \text{and} \quad \tilde{t}_0=\frac{t_0}{b},
\end{equation}
is a solution of \eqref{original} for $0<t<\tilde{t}_0$.
In obvious notation
\begin{equation} \label{subst (2)}
     \underline{\tilde{\varphi}}(t)=\begin{bmatrix}
\tilde{u}_1(t) \\ \tilde{u}_2(t) \\ \tilde{w}_1(t) \\ \tilde{w}_2(t) \end{bmatrix} =
\begin{bmatrix}  a u_1(bt) \\ ab u_2(bt) \\ a^{p-1}b^{2p-2}w_1(bt) \\ a^{p-1} b^{2p-1} w_2(bt) \end{bmatrix} 
\end{equation}
is the corresponding solution of $(*)$ so that 
\begin{equation} \label{subst (3)}
     \underline{\tilde{\varphi}}(0)=\begin{bmatrix}
0 \\ ab \alpha \\ 0 \\ a^{p-1} b^{2p-1} \beta \end{bmatrix}=
\begin{bmatrix} 0 \\ \tilde{\alpha} \\ 0 \\ \tilde{\beta} \end{bmatrix}.
\end{equation}
\end{Remark}

We close this section by addressing local and global uniqueness of the solution. The function $F(\lambda;\cdot):\mathbb{R}^4\longrightarrow \mathbb{R}^4$ is continuous. Therefore, by the Cauchy-Peano Theorem, for any $\underline{\psi}\in \mathbb{R}^4$ and $t_1\geq 0$, there exists $\delta>0$ such that \eqref{vector} has a solution $\underline{\varphi}\in C^1([t_1,t_1+\delta])^4$ and $\underline{\varphi}(t_1)=\underline{\psi}$. 

The function $F(\lambda;\cdot)$ is Lipschitz for all $x\not=0$ and $z\not=0$, therefore the solution will be locally unique by the classical Cauchy-Lipschitz Uniqueness Theorem, for $u_1(t_1)\not=0$ and $w_1(t_1)\not=0$. If $p>2$ or $q<2$,  $F(\lambda;\cdot)$ is not Lipschitz at $x=z=0$. However, remarkably, the solution turns out to be locally unique also, if and only if $\underline{\psi}\not=0$. This discovery dates back to the work of Benedikt. The next statement summarizing this fact follow directly from \cite[Proposition~3.1.5]{Be2004}. 

\begin{Lemma} \label{globaluniqueness}
Let $1<p,q<\infty$ and $\lambda>0$ be fixed. Let $t_1\geq 0$ and $\underline{\psi}\in \mathbb{R}^4\setminus\{0\}$. There exists $\delta>0$, such that \eqref{vector} has a unique solution $\underline{\varphi}\in C^1([t_1,t_1+\delta])^4$ satisfying $\underline{\varphi}(t_1)=\underline{\psi}$. 
\end{Lemma}

By standard continuation arguments, the above implies that there exist maximal $t_{\infty}\leq \infty$, dependants on $(\alpha,\beta)$, such that the solution $\underline{\varphi}\in C^1(0,t_{\infty})$ is unique, for any given $(\alpha,\beta)\not=(0,0)$. In order to show the uniqueness in Theorem~\ref{existence}, we will see below that, subject to the boundary conditions of \eqref{original}, $\underline{\varphi}$ is unique in the case $\underline{\psi}=0$ for any $t_1\in (0,t_0)$. See Lemma~\ref{leadingtouniqueness}.

The next result, about the expected property of continuity of the solutions with respect to $(\alpha,\beta)$, is a consequence of combining the uniqueness with classical statements such as \cite[Theorem~4.3]{CL1955}. The proof is straightforward. 

\begin{Lemma}  \label{continuity}
Let $(\alpha_1,\beta_1)\in (0,\infty)^2$. Let $\underline{\varphi}^1\in C^1(0,t_\infty)$ be the solution of \eqref{vector},  unique for $t_{\infty}\leq \infty$, such that $\underline{\varphi}^1(0)=[ 0,\, \alpha_1 ,\, 0 ,\, \beta_1 ]^T$. Let $t_1\in(0,t_{\infty})$. Then, there exists $\delta>0$ such that for any \[(\alpha,\beta)\in(\alpha_1-\delta,\alpha_1+\delta)\times (\beta_1-\delta,\beta_1+\delta),\] 
all solutions $\underline{\varphi}$ of \eqref{vector} satisfying
$\underline{\varphi}(0)=[ 0,\, \alpha ,\, 0 ,\, \beta ]^T$
exist over $[0,t_1]$.
Moreover, 
\[
     \lim_{(\alpha,\beta)\to (\alpha_1,\beta_1)}\underline{\varphi}(t)=\underline{\varphi}^1(t)
\]
 uniformly for all $t\in[0,t_1]$. 
\end{Lemma}

For some pairs $(p,q)$ and $(\alpha,\beta)$, concretely for $p\geq q$ and any $(\alpha,\beta)\in \mathbb{R}^2$, the solution of $(*)$ exists for all $t\in[0,\infty)$. However, for certain combinations of these parameters, the solution develops singularities at some $t<\infty$. The next two lemmas describe the oscillatory behaviour of these solutions, in both cases.

\begin{Lemma} \label{specificlongtime}
Let $\lambda,\,p,\,q,\,\alpha=u_2(0)$ and $\beta=w_2(0)$, be such that there exists a solution $\underline{\varphi}\in C^1(0,\infty)^4$ of $(*)$. If one of the components of $\underline{\varphi}$ is uniformly bounded in $t\in[0,\infty)$, then all the components of $\underline{\varphi}$ have infinitely many zeros. 
\end{Lemma}
\begin{proof}
Without loss of generality, suppose that it is $u_1(t)$ the one component that is uniformly bounded. That is,
\[
    \limsup_{t\to \infty} |u_1(t)|<\infty.
\] Then, note that 
\[\liminf_{t\to \infty} |u_1(t)|=0.\] Indeed, the fact that
\[\liminf_{t\to \infty} |u_1(t)|>0,\]
would imply that 
\[
    \lim_{t\to \infty}|w_2(t)|=\infty,
\]
which by following the connections of \eqref{integral} between the different components of $\underline{\varphi}(t)$, would lead to 
\[
    \lim_{t\to \infty}|u_1(t)|=\infty
\]
creating a contradiction. Moreover, in fact
\begin{equation} \label{onlyposlimit}
        \liminf_{t\to \infty} |u_k(t)|= 
        \liminf_{t\to \infty} |w_k(t)|=0 \qquad k=1,2,
\end{equation}
also, for otherwise we end up with the same contradiction.

Now we know that \eqref{onlyposlimit} holds true and two possibilities arise. One is that the modulus of one of the components of $\underline{\varphi}(t)$ has non-zero limsup at $t\to \infty$. That is,
\[
     \limsup_{t\to \infty} \big(|u_1(t)|+|u_2(t)|+|w_1(t)|+|w_2(t)|\big)>0.
\]
In this case, because of \eqref{onlyposlimit} and by following the connections between the different entries of $\underline{\varphi}(t)$ via \eqref{integral}, we gather that the next component will necessarily have infinitely many zeros. In turn all components will have infinitely many zeros, reaching the conclusion of the lemma.

The other possibility is that all components of $\underline{\varphi}(t)$ have zero limit as $t\to \infty$. That is,
\[
     \lim_{t\to \infty} \big(|u_1(t)|+|u_2(t)|+|w_1(t)|+|w_2(t)|\big)=0.
\]
Here the limit exists, because the limsups and liminfs of all the terms coincide and are all equal to zero. This being the case, if one entry of  $\underline{\varphi}(t)$ vanishes at time $t_1>0$, then there exists $t_2>t_1$ such that the next entry (counting cyclically) vanishes at $t_2$. Arguing recursively, this implies that all entries of $\underline{\varphi}(t)$ should have infinitely many zeros as claimed in the conclusion of the lemma. 
\end{proof}

We will show in Section~\ref{section4} a sharper result than this lemma in the context of the eigenvalue equation \eqref{original}. Namely, if $u_1(t_1)=w_1(t_1)=0$ for some $t_1>0$, then all components of the solution are uniformly bounded, periodic and symmetric. 

We now consider the behaviour of the solution near finite singular points.

\begin{Lemma} \label{behasol}
Let $\lambda,\,p,\,q,\,\underline{\zeta}$ be such that, there exists a  solution vector $\underline{\varphi}\in C^1(0,t_{\infty})^4$ of $(*)$ for some $0<t_{\infty}<\infty$. If
\begin{equation*} \label{forsome}
   \limsup_{t\to t_{\infty}^-} v(t)=\infty \qquad\text{for one of }v\in\{u_1,u_2,w_1,w_2\},
\end{equation*}
then
\begin{equation} \label{forall}
   \left|\liminf_{t\to t_{\infty}^-} v(t)\right|=\left|\limsup_{t\to t_{\infty}^-} v(t)\right|=\infty \qquad\text{for all }v\in\{u_1,u_2,w_1,w_2\}.
\end{equation}
\end{Lemma}
\begin{proof}
Without loss of generality we assume that, 
\begin{equation*}
   \limsup_{t\to t_{\infty}^-}u_{1}(t)=\infty.
\end{equation*}   
There are two possibilities to consider. 

One possibility is that
\[
     0\leq \liminf_{t\to t_{\infty}^-}u_1(t)\leq \infty.
\]
In that case, according of \eqref{integral}, near $t_\infty$ we should have $w_2(t)$ monotonic decreasing, hence $w_1(t)$ monotonic decreasing and $u_2(t)$ is monotonic increasing. Then, $u_1(t)$ is monotonic increasing and positive near $t_\infty$. Therefore, in fact
\[
    \lim_{t\to t_\infty^-} u_1(t)=\infty.
\]
But then, from the formulation \eqref{differential}, it follows that 
\[
      \lim_{t\to t_\infty^-}u_2(t)=\infty
\qquad \text{and also that} \qquad      \lim_{t\to t_\infty^-}w_k(t)=-\infty.
\]
This implies \eqref{forall}.

The other possibility is that
\[
    \liminf_{t\to t_\infty^-} u_1(t)<0.
\]
From the formulation \eqref{differential} it then follows that
\[
    \liminf_{t\to t_\infty^-} u_2(t)=-\infty \qquad \text{and} \qquad
     \limsup_{t\to t_\infty^-} u_2(t)=+\infty,
\] 
as $u_1(t)$ becomes highly oscillatory at $t_\infty<\infty$ with negative minima and positive maxima. Likewise, and for similar reasons, also
\[
    \liminf_{t\to t_\infty^-} w_k(t)=-\infty \qquad \text{and} \qquad
     \limsup_{t\to t_\infty^-} w_k(t)=+\infty.
\] 
Then, 
\[
    \liminf_{t\to t_\infty^-} u_1(t)=-\infty.
\] 
Therefore, once again, we have \eqref{forall}.
\end{proof}

\begin{Remark} \label{rem1}
When a solution exists in a segment $(0,t_\infty)$ and has a singularity at $t_\infty$, we have the following assertion. If one of the components of the solution vector does not have a zero in $t\in(t_1,t_\infty)$ for some $0<t_1<t_\infty$, then all the components of the solution vector are monotonic for $t\in (t_2,t_\infty)$ where $t_2\geq t_1$ is large enough. Moreover, in that case
\begin{equation} \label{theimportant}
     \lim_{t\to t_{\infty}^-}u_k(t)=-\lim_{t\to t_{\infty}^-}w_k(t).
\end{equation}
The proof of this is identical to that of the first possibility in the proof of Lemma~\ref{behasol}.
\end{Remark}

The solution does not become oscillatory near a singularity $t_\infty<\infty$ for a range of parameters $(p,q)$. Namely 
\[
    \lim_{t\to t_\infty^-}|u_k(t)|=\lim_{t\to t_\infty^-}|w_k(t)|=\infty
\]
and \eqref{theimportant} holds true. In fact we conjecture that the latter is the case for all $p>1$ and $q>1$, but we are not currently able to complete the proof of this claim. We will not need this fact below.


\section{Stability of solutions and proof of existence} \label{section3}
We now settle the existence part of Theorem~\ref{existence}. The first statement below, about the monotonicity of the solutions of $(*)$ in terms of the initial data, will be combined with Lemma~\ref{continuity} in order to construct solutions whose components develop zeros in $(0,\infty)$. These solutions will then be further perturbed and dilated, to match the boundary condition at $t=t_0$. 

\begin{Lemma} \label{lemaA}
Let $\lambda>0$ and $p,\,q>1$ be fixed. Consider the evolution systems $(*)$-\eqref{initialstate}. Let $0<\alpha_2\leq \alpha_1$, $0<\beta_1\leq \beta_2$ and assume that at least one of these inequalities is strict. Let $t_1>0$ be such that all 
the components of the solution\footnote{Here and everywhere below, the indices $j$ (on top) refer to corresponding sub-indices of $\alpha$ or $\beta$, in context.} $|u_k^j(t)|$ and $|w_k^j(t)|$ are finite for $t\in (0, t_1]$. Then,
 \begin{equation} \label{mono}
    u^2_k(t)< u_k^1(t) \quad \text{and} \quad
w_k^1(t)< w_k^2(t), \qquad \text{for } k=1,2\quad \text{and} \quad t\in (0,t_1].
\end{equation}
Moreover, 
\begin{equation}    \label{rate}
\begin{aligned} 
    u_1^1(t)-u_1^2(t)&> (\alpha_1-\alpha_2) t \\
    w_1^2(t)-w_1^1(t)&> (\beta_2-\beta_1) t 
\end{aligned}
\end{equation}
for all $t\in (0,t_1]$.
\end{Lemma}
\begin{proof}
As the proofs of the two cases are almost identical, without loss of generality we assume that $\alpha_2\leq \alpha_1$ and $\beta_1<\beta_2$. 

By virtue of  \eqref{integral} and by monotonicity of each one of the integrals in this formulation, it follows that there exists $\varepsilon>0$, such that the left hand side of \eqref{mono} holds true for all $t\in (0,\varepsilon]$.  Assume that $t$ lies in this segment. 
Then, 
\begin{equation}\label{non-strictfirst}
\begin{aligned}
    w_2^2(t)-w_2^1(t)&=\beta_2-\beta_1+\lambda \int_0^t \Big(\spower{u^1_1(s)}{q}-\\ & \hspace{3cm} \spower{u_1^2(s)}{q}\Big)\mathrm{d}s\\ & \geq \beta_2-\beta_1>0. 
\end{aligned}
\end{equation}
Hence,
\begin{equation} \label{non-strictsecond}
    w_1^2(t)-w_1^1(t)=\int_0^t (w_2^2(t)-w_2^1(t))\mathrm{d}s\geq(\beta_2-\beta_1)t. 
\end{equation}
Thus,
\begin{align*}
     u_2^1(t)&-u_2^2(t)\\ & =\alpha_1-\alpha_2+\int_{0}^t \left(\spower{w_1^2(s)}{p'}-\spower{w_1^1(s)}{p'}\right)\mathrm{d}s\\&>\alpha_1-\alpha_2\geq 0. 
\end{align*}
Note that the strict inequality here follows from the previous inequality and the fact that then the integral on the right hand side is strictly increasing for the stated values of $t$.
Then, 
\[
     u^1_1(t)-u^2_1(t)= \int_{0}^t (u_2^1(s)-u_2^2(s))\mathrm{d}s>(\alpha_2-\alpha_1)t.
\]
But then, going back to \eqref{non-strictfirst} with the latter, it follows that in fact
\[
    w_2^2(t)-w_2^1(t)>\beta_2-\beta_1.
\]
And, carrying on to \eqref{non-strictsecond}, we also have 
\[
    w_1^2(t)-w_1^1(t)>(\beta_2-\beta_1)t.
\]

As a conclusion of the previous paragraph, so far we now know that \eqref{mono} and \eqref{rate} hold true for $t\in(0,\varepsilon]$ with strict inequality. Let $\varepsilon\leq \varepsilon_1\leq t_1$ be the maximum $t$ such that \eqref{mono} and \eqref{rate} hold true for $t\in(0,\varepsilon_1]$ with strict inequality. As we can run an analogue of the above argument from $t=\varepsilon_1$, with suitable constants in front of all the integrals in \eqref{integral},  then necessarily $\varepsilon_1=t_1$. This completes the proof of the statement.
\end{proof}

The next remark will be repeatedly used below for the case $\varphi_j=u_1,\,w_1$ and $\varphi_k=w_1,\,u_1$.  

\begin{Remark} \label{useful}
If the entry $\varphi_j(t)$ of a solution vector $\underline{\varphi}(t)$ has 3 simple zeros at $r_0<r_1<r_2$, where $r_0\geq 0$, then
necessarily that entry should have inflection points lying between $r_0$ and $r_2$. Hence, for $k\equiv j+2 \mod 4$, the entry $\varphi_{k}$ should have at least one zero at a point $s_0\in(r_0,r_2)$ .
\end{Remark}

We now determine a mechanism for perturbing $(\alpha,\beta)$ in order to create zeros for the different components of the solution vector.

\begin{Corollary} \label{lemaB}
Let $\lambda>0$ and $p,\,q>1$ be fixed. Consider the solution to the systems $(*)$-\eqref{initialstate}. There exist $\alpha>0$ and $\beta>0$, such that both $u_1(t)$ and $w_1(t)$ vanish at least once for $t>0$. Moreover, we can find a pair $(\alpha,\beta)\in(0,
\infty)^2$, such that one of these two functions has at least two zeros, $0<t_1< t_2$, the other has one zero, $r_1\in(0,t_2)$, and these are the first zeros counting from the left of the respective functions. 
\end{Corollary}
\begin{proof}
We show the first claim. The other claims follow immediately from Remark~\ref{useful} taking $r_0=0$. Moreover, for the rest of the proof, we assume without loss of generality that $\alpha>0$ and $\beta>0$ are such that $u_1^1(t)\equiv u_1(t)$ has no positive zeros and is eventually monotonic increasing, while $w_1^1(t)\equiv w_1(t)$ has only one positive zero and is eventually monotonic decreasing. We are going to invoke Lemma~\ref{lemaA}, therefore let us set $(\alpha_1,\beta_1)=(\alpha,\beta)$.

Let $t_1>0$ be such that $w_1^1(t_1)=0$ and $w_1^1(t)>0$ for all $t\in(0,t_1)$. Then $(u_2^1)'(t_1)=0$ and $(u_2^1)'(t)<0$ for all $t\in(0,t_1)$. Hence $u^1_1(t)$ is concave for all $t\in (0,t_1)$. Let
\[
    0<\alpha_2=\frac{\alpha_1 t_1-u^1_1(t_1)}{t_1}.
\]  
Then $\alpha_2>0$  and also $\alpha_2<\alpha_1$. From Lemma~\ref{lemaA}, it follows that 
\[
      u_1^1(t_1)-u_1^2(t_1)> (\alpha_1-\alpha_2)t_1=u_1^1(t_1).
\]
Hence $u_1^2(t_1)< 0$, so this component must have a zero below $t_1$, and the proof now branches into four possibilities. 

One possibility is that either $u_1^2(t)$ or $w_1^2(t)$ are uniformly bounded in $(0,t_\infty)$. By virtue of Lemma~\ref{behasol}, then $t_\infty=\infty$. In this case, the conclusion follows from Lemma~\ref{specificlongtime}. Another possibility is that all $u_k^2(t)$ and $w_k^2(t)$ become unbounded oscillatory near $t_\infty<\infty$, see Lemma~\ref{lemaA} and its proof. But then, once again the conclusion follows. A third possibility is that the $u_k^2(t)$ point upwards (the limits at $t_\infty$ are $+\infty$), $w_k^2(t)$ point downwards and once again the conclusion follows taking $(\alpha,\beta)=(\alpha_2,\beta_1)$, because then $w^2_1(t)$ will also have a zero. 

The fourth possibility is that for the pair $(\alpha_2,\beta_1)$, $u_k^2(t)$ point downwards and $w_k^2(t)$ upwards at $t_\infty$. At this point it is not guaranteed that $w_1^2(t)$ has a positive zero. Let
\begin{align*}
   \alpha_3&=\inf\{\alpha\leq \alpha_1\,:\, \lim_{t\to t_\infty^-} u_1(t)=+\infty\} \\
   \alpha_4&=\sup\{\alpha\geq \alpha_2\,:\, \lim_{t\to t_\infty^-} u_1(t)=-\infty\}.
\end{align*}
If $\alpha_4<\alpha_3$, then for $\alpha=\frac{\alpha_3+\alpha_4}{2}$ the component $u_1(t)$ is either bounded or oscillatory as $t\to t_\infty^-$. Hence, the conclusion follows as in the previous first or second possibilities. If, on the other hand, $\alpha_3=\alpha_4$, call  $u_1^3(t)\equiv u_1(t)$ the first component of the solution vector for the pair $(\alpha_3,\beta_1)$. We consider three further sub-cases. 

If both
\begin{align*}
   \alpha_3\not \in \{\alpha\leq \alpha_1\,:\, \lim_{t\to t_\infty^-} u_1(t)=+\infty\} \qquad \text{and}\\
   \alpha_4\not\in \{\alpha\geq \alpha_2\,:\, \lim_{t\to t_\infty^-} u_1(t)=-\infty\},
\end{align*}
then $u_1^3(t)$ is either bounded or oscillatory as $t\to t_\infty^-$ and once again the conclusion follows. 
If, by contrast,
$
  \lim_{t\to t_\infty^-} u^3_1(t)=+\infty,
$
then $u^3_1(t)$ has an inflection point at the positive zero of $w^3_1(t)$. According to Lemma~\ref{continuity}, for $\delta>0$ small enough, the component $w^5_1(t)$ corresponding to the pair $(\alpha_5=\alpha-\delta,\beta_1)$ would necessarily have two zeros above this inflection point. Indeed, $u^5_1(t)$ would also have an inflection point close the one of $u^3_1(t)$ plus an extra inflection point, due to the fact that   $
  \lim_{t\to t_\infty^-} u_1^5(t)=-\infty.
$
According to Remark~\ref{useful}, then $u_1^5(t)$ has also a positive zero. Therefore, the conclusion follows for this sub-case as well. The final sub-case, involving the condition
$
  \lim_{t\to t_\infty^-} u^3_1(t)=-\infty,
$ 
can be dealt with in a similar manner.
\end{proof}

Corollary~\ref{lemaB} gives four possibilities which we label for later use.
\begin{enumerate}[I.]
\item \label{I} $u_1(t)$ is the one with two zeros, $t_1<t_2$, $w_1(t)$ the one with at least one zero $r_1<t_2$, and
\begin{enumerate}[a)]
\item \label{I.a} either $r_1\geq t_1$ 
\item \label{I.b} or $r_1< t_1$.
\end{enumerate}
\item \label{II} (Roles of $u_1$ and $w_1$ swapped), $w_1(t)$ is the one with two zeros, $r_1<r_2$, $u_1(t)$ the one with at least one zero $t_1<r_2$, and
\begin{enumerate}[a)]
\item \label{II.a} either $t_1\geq r_1$ 
\item \label{II.b} or $t_1< r_1$.
\end{enumerate}
\end{enumerate}  

Now we turn to solution vectors to $(*)$-\eqref{initialstate} whose first and third components vanish at $t_0>0$. The next lemma, whose proof reduces to describing the interlacing of the zeros between the different derivatives of a solution to \eqref{original}, is expected but it will be crucial in what follows. 

\begin{Lemma} \label{interlace}
Let $n\geq 0$. Let $u(t)$ be an eigenfunction of \eqref{original}. Then, $u(t)$ has exactly $n$ zeros counting multiplicity in the segment $(0,t_0)$ if and only if $u''(t)$ has exactly $n$ zeros counting multiplicity in  $(0,t_0)$. 
\end{Lemma}
\begin{proof}
Suppose that $u(t)$ has $n$ zeros in $(0,t_0)$ counting multiplicity. Since $u(0)=u(t_0)=0$ then,  $u'(t)$ has at least $n+1$ zeros in $(0,t_0)$. Then, $u''(t)$ must have at least $n$ zeros in $(0,t_0)$, because it must vanish in between two consecutive zeros of $u'(t)$ taking into account multiplicity. So, in the notation of the systems $(*)$, $w_1(t)$ has at least $n$ zeros counting multiplicity in $(0,t_0)$.

Now, if $w_1(t)$ has $n+1$ zeros in $(0,t_0)$ or more, arguing cyclically, we would then have that $u(t)=u_1(t)$ has $n+1$ or more zeros in $(0,t_0)$. Since this is not the case, it follow that $u''(t)$ must have exactly $n$ zeros in $(0,t_0)$.

The proof of the converse is identical.
\end{proof}

\begin{Corollary} \label{corollary1}
Let $\lambda>0$ and $p,\,q>1$ be fixed. There exist $\alpha>0$ and $\beta>0$, such that, for the solution vector of $(*)$-\eqref{initialstate}, the first and third components $u_1(t)$ and $w_1(t)$, both have a simple zero at the same point $t_1>0$ and both are strictly positive for $t\in(0,t_1)$. 
\end{Corollary}
\begin{proof}
By virtue of Lemma~\ref{interlace}, it is sufficient to show that there exists a pair $(\alpha,\beta)\in(0,\infty)^2$, such that at some $t_1>0$ both components of the solution vector vanish, $u_1(t_1)=w_1(t_1)=0$, and only one is positive up to $t_1$, say, $u_1(t)>0$ for all $t\in(0,t_1)$.

Let a pair $(\alpha_0,\beta_0)\in(0,\infty)^2$ be such that one of the possibilities of Corollary~\ref{lemaB} holds true. The proof branches accordingly.

Suppose that \ref{I}\ref{I.a} occurs. Fix 
\[
      d=\left| \min_{t_1\leq t \leq t_2} u^0_1(t)\right|>0 \quad \text{and} \quad \alpha_1= \frac{d}{t_1}+\alpha_0.   
\] 
Combining lemmas~\ref{continuity} and \ref{lemaA}, we gather that, if the parameter $\alpha$ is varied continuously from $\alpha_0$ to $\alpha_1$ and $\beta_0$ remains fixed, then  the component $u_1(t)$ of the solution increases, the difference between the two zeros $t_1-t_2$ decreases, $w_1(t)$ decreases, and $r_1$ begins its trajectory lying between $t_1$ and $t_2$.  At $\alpha_1$, according to Lemma~\ref{lemaA}, 
\[
u^1_1(t)-u^0_1(t)\geq (\alpha_1-\alpha^0) t \geq (\alpha_1-\alpha^0) t_1 \geq d
\]
for all $t_1<t<t_2$. Then, $u_1^1(t)>0$ for all $t\in (0,t_2)$. Therefore, as $t_k$ and $r_1$ are continuous functions of $\alpha$, there is an intermediate value $\alpha=\alpha_2$ such that $u_1(t_1)=u_1(t_2)=0$, $u_1(t)>0$ for all $t\in (0,t_1)$, and $w_1(r_1)=0$ for $r_1\in\{t_1,t_2\}$. If $r_1=t_1$ the conclusion follows and there is nothing else to prove. So, assume $r_1=t_2\not=t_1$. Then, $w_1(t)\not=0$ for $t\in(0,t_1]$. Moreover, by virtue of Lemma~\ref{interlace} case $n=1$, for the pair $(\alpha_2,\beta_0)$ there is a zeros of $w_1$ lying in the segment $(t_1,t_2)$. Then, by increasing $\alpha$ further, and appealing to continuity, we eventually reach a point in the deformation where we have $w_1(t)>0$ for all $t\in(0,r_1)$, $r_1=t_1$ and hence the claimed conclusion.

Suppose now that it is the possibility \ref{I}\ref{I.b} the one that holds. Let $d=|w_1^0(t_1)|=-w_1^0(t_1)$. If $\beta$ increases from $\beta_0$, then $r_1$ increases, and $t_1$ and $t_2$ decrease. Let $\beta_1$ be such that
\[
     (\beta_1-\beta_0)t_1\geq d.
\]
Then, by Lemma~\ref{lemaA}, 
\[
      w_1^1(t_1)-w_1^0(t_1)\geq (\beta_1-\beta_0)t_1\geq d.
\]
Hence $w_1^1(t_1)>0$. Therefore, by continuity, there should have been a point in the deformations of $u_1(t)$ and $w_1(t)$, where $\beta$ was such that $r_1=t_1$.

Finally, if the cases \ref{II}) occurs, we argue by swapping the roles of the components $u_1(t)$ and $w_1(t)$.
\end{proof}

The existence part of Theorem~\ref{existence} follows directly from this corollary for $t_1=t_0$. But note that $t_1$ depends on $p$, $q$ and $\lambda$. We therefore combine this with the rescaling \eqref{subst (1)} to complete the proof for general $t_0>0$. 

\begin{proof}[Proof of existence in Theorem~\ref{existence}]
There are two statements to show. 

For the first statement, let $p\not=q$ and $\lambda>0$. According to Corollary~\ref{corollary1}, there exists $t_1>0$ such that the corresponding first component of the solution vector $u_1(t)$ and its second derivative $u_1''(t)$ are positive in $(0,t_1)$ and they both vanish at $t=0$ and  $t=t_1$. This shows that there exists a solution $u(t)=u_1(t)$ to \eqref{original}, whenever $t_0=t_1$.

Now let $t_0>0$ be arbitrary. Take 
\[
     b=\frac{t_0}{t_1} \qquad \text{and} \qquad a=\left(\frac{t_1}{t_0}\right)^{\frac{2p}{p-q}}
\] 
in the substitution \eqref{subst (1)}--\eqref{subst (2)}. 
Then, $u(t)=au_1(bt)$ is a positive solution of \eqref{original} for eigenvalue $\tilde{\lambda}=\lambda a^{p-q}b^{2p}=\lambda$. This ensures the first statement of the existence claim in Theorem~\ref{existence}.

For the second statement, let $p=q$. By Corollary~\ref{corollary1}, for eigenvalue parameter $\lambda=1$, there exists $t_1>0$ such that the corresponding $u_1(t)>0$ for all $t\in(0,t_1)$ solves \eqref{original} for $t_0=t_1$. Take
\[
     b=\frac{t_0}{t_1} \qquad \text{and} \qquad a=1.
\] 
Then $u(t)=\widetilde{u_1}(t)$ in \eqref{subst (1)}--\eqref{subst (2)}, is a positive solution of \eqref{original} for $t\in(0,t_0)$ and eigenvalue
\[
    \lambda=\left(\frac{t_0}{t_1}\right)^{2p}.
\]  
This is the second statement of the existence claim in Theorem~\ref{existence}.
\end{proof}


\section{Symmetries, periodicity and uniqueness} \label{section4}

In this section we complete the proof of Theorem~\ref{existence}. We begin by showing that, for fix $p,q>1$ and $\lambda>0$, an eigenfunction $u(t)$ of \eqref{original} is unique for any choice of $(\alpha,\beta)\in [0,\infty)^2$. Here we include the statement that, the only solution in the case $(\alpha,\beta)=(0,0)$ is the trivial solution $u(t)=0$. More generally, as we establish next,  the only solution to $(*)$ satisfying the boundary conditions of \eqref{original} for which all components vanish simultaneously at a point, is the trivial solution.

\begin{Lemma} \label{leadingtouniqueness}
Let $\underline{\varphi}\in C^1([0,t_0])^4$  be a solution to $(*)$, such that $\underline{\varphi}(0)=[0,\alpha,0,\beta]^T$ and $\underline{\varphi}(t_0)=[0,\tilde{\alpha},0,\tilde{\beta}]^T$. Then, $\alpha\beta\geq 0$ and $\tilde{\alpha}\tilde{\beta}\geq 0$. Moreover, $\underline{\varphi}(t)=0$ for some $t\in[0,t_0]$ if and only if  $\underline{\varphi}(t)=0$ for all $t\in[0,t_0]$. 
\end{Lemma}
\begin{proof}
From the formulation \eqref{integral} follows that, if $\alpha$ and $\beta$ have different signs, then $\varphi_3(t)$ is of opposite sign to $\varphi_1(t)$, and so $\varphi_1(t)$ is monotonic and it therefore cannot vanish at $t=t_0$. Similarly, $\tilde{\alpha}\tilde{\beta}<0$ implies that $\varphi_3(0)\not=0$.
 This shows the first claim. For the second claim, the argument similar to the one used in the proof of Lemma~\ref{interlace}, is as follows.

Consider first the case $t=0$. Assume that $\alpha=\beta=0$. Let $u_1(t)$ have $n$ zeros (counting multiplicity) in $(0,t_0)$. Then, $u_2(t)$ has at least $n+1$ zeros in $(0,t_0)$. But, since $u_2(0)=0$, we have that $w_2(t)$ has at least $n+1$ zeros in $(0,t_0)$. This contradicts the statement of Lemma~\ref{interlace}. Hence, necessarily $\alpha\beta\not=0$.

Secondly, for $t=t_0$, apply the result already proven to the solution $\tilde{\underline{\varphi}}(t)=  \underline{\varphi}(t_0-t)$, in order to get that $\tilde{\alpha}\tilde{\beta}\not=0$.

Finally, let $t=t_1\in(0,t_0)$. The solution to $(*)$ in $[t_1,t_0]$ satisfies the boundary conditions of \eqref{original} at $t=t_1$ and $t=t_0$. That is, the vector $\tilde{\underline{\varphi}}(t)=\underline{\varphi}(t-t_1)$ is a solution to $(*)$ in $C^1([0,t_0-t_1])^4$, satisfying the boundary conditions of \eqref{original} at the end point of the segment. Then, apply the first case to obtain that  the only possibility is that $\varphi_2(t_1)\varphi_4(t_1)\not=0$.
\end{proof}

\begin{Corollary} \label{Unique Theorem}
Let $n\geq 0$. Let $u(t)$ and $\tilde{u}(t)$ be two eigenfunctions of \eqref{original} associated with eigenvalues $\lambda>0$ and $\tilde{\lambda}>0$ respectively, such that they both have $n$ zeros in $(0,t_0)$. Then, $\tilde{u}(t)=\pm au(t)$ and $\tilde{\lambda}=\lambda a^{p-q}$ for some $a>0$. 
\end{Corollary}
\begin{proof}
\underline{Case $p\not=q$}. According to Lemma~\ref{leadingtouniqueness}, without loss of generality, we can assume that both $u(t)$ and $\tilde{u}(t)$ have first and third derivatives positive at $t=0$. We prove the claim of the corollary by contradiction. Taking $b=1$ and a suitable $a>0$ in \eqref{subst (1)}, the negation of the statement to be shown, is equivalent to assuming that, for the same fixed $\lambda>0$, there exists two solutions of \eqref{original} on $(0,t_0)$ with $n$ zeros, say $u^1(t)\not=u^2(t)$. We therefore suppose this and derive a contradiction.

Let $\alpha_k=u_2^k(0)>0$ and $\beta_k=w^k_2(0)>0$. Combining 
lemma~\ref{globaluniqueness} and \ref{leadingtouniqueness}, follows that, either $\alpha_1 \not = \alpha_2$ or $\beta_1 \not = \beta_2$. Without loss of generality, we further assume that $\alpha_1\not=\alpha_2$. The other case is similar and leads to the same conclusion. 

Let $a,b>0$ be such that 
\[ a^{p-q}b^{2p}=1 \qquad \text{and} \qquad ab=\frac{\alpha_1}{\alpha_2}.\] 
Then necessarily $b\not=1$.
Let \[\hat{\alpha}=ab \alpha_2=\alpha_1 \qquad \text{and} \qquad \hat{\beta}=a^{p-1}b^{2p-1}\beta_2.\] 
Hence, $\hat{u}(t)=au^2(bt)$ is a solution of \eqref{original} in $(0,\hat{t}_0)$ for the same eigenvalue $\lambda>0$ but end point $\hat{t}_0=\frac{t_0}{b}\not=t_0$. Moreover, $\hat{u}(t)$ and $\hat{u}''(t)$ have $n$ zeros in $(0,\hat{t}_0)$. The corresponding solution vector $\underline{\hat{\varphi}}(t)$ satisfies \eqref{integral} with $\hat{u}_2(0)=\alpha_1$ and $\hat{w}_2(0)=\hat{\beta}$, for which $\hat{u}_1(t)$ and $\hat{w}_1(t)$ have $n$ zeros in $(0, \hat{t}_0)$. Additionally they are such that 
\begin{equation} \label{keytrick}
\hat{u}_1(0)=\hat{w}_1(0)=\hat{u}_1(\hat{t}_0)=\hat{w}_1(\hat{t}_0)=0. 
\end{equation}

Two possibilities now arise. Either $\beta_2<\hat{\beta}$ or $\beta_2>\hat{\beta}$. But these two possibilities are incompatible with Lemma~\ref{lemaA}. Indeed, the first possibility renders $\hat{u}_1(t)<u^2_1(t)$ and $w^2_1(t)<\hat{w}_1(t)$ for all $t>0$ whereas the second possibility renders $\hat{u}_1(t)>u^2_1(t)$ and $w^2_1(t)>\hat{w}_1(t)$ for all $t>0$. Either case, we have that either $\hat{u}_1(\hat{t}_0)\not=0$ or $\hat{w}_1(\hat{t}_0)\not=0$ or that one of these functions does not have exactly $n$ zeros in $(0,\hat{t}_0)$. Any of this, contradicts what we know already about $\hat{u}(t)$. Hence, we must have $b=1$ and so $u^2(t)=\hat{u}(t)=u^1(t)$. From this, the conclusion of the lemma for $p\not=q$ follows.

\underline{Case $p=q$}. Just as before, suppose that for the same $\lambda>0$ we have two positive solutions of \eqref{original} with $n$ zeros, $u^1(t)$ and $u^2(t)$. Let
$c=\frac{\alpha_2}{\alpha_1}$, where $\alpha_k$ are the second components of the corresponding solution vectors. Set $\hat{u}(t)=cu^1(t)$. As the equation is homogeneous, then $\hat{u}(t)$ is also a solution for the same eigenvalue $\lambda$. 

After this we can proceed in similar way as the case $p\not=q$. We have $\hat{u}_2(0)=\alpha_2=u^2_2(0)$ and, by virtue of Lemma~\ref{lemaA}, the only possibility not leading to a contradiction is that also $\hat{w_2}(0)=\beta_2=w^2_2(0)$. But then, combining lemmas~\ref{globaluniqueness} and \ref{leadingtouniqueness}, ensures that
\[
     u^2(t)=\hat{u}(t)=c u^1(t).
\]
\end{proof}

We now complete the proof of Theorem~\ref{existence}. 

\begin{proof}[Proof of uniqueness in Theorem~\ref{existence}]
The case $p\not=q$ is a direct consequence of Corollary~\ref{Unique Theorem} for $n=0$. 

Let $p=q$ instead. By virtue of Corollary~\ref{Unique Theorem}, we only need to show that $\lambda$ is unique. Suppose that we have two $\lambda_1\not=\lambda_2$ with corresponding eigenvectors $u^1(t)$ and $u^2(t)$, positive on $(0,t_0)$ and vanishing at the end points alongside their second derivatives. Set
\[
     b=\left(\frac{\lambda_2}{\lambda_1}\right)^{\frac{1}{2p}}
\] 
and consider 
$u^3(t)=u^1(bt)$. Then $\underline{\varphi}^2(t)$ and $\underline{\varphi}^3(t)$ are both solutions vectors of the systems $(*)$ for the same $\lambda\equiv \lambda_2=\lambda_1 b^{2p}$. 

Now, let 
\[
     a=\frac{u^3_2(0)}{u^2_2(0)}.
\] 
Then, $au^2_2(0)=u^3_2(0)=\alpha_3$, that is, both $au^2(t)$ and $u^3(t)$ match their derivatives at $t=0$. Now the third derivative satisfies the following three possibilities. If $\beta_2=\beta_3$, then $u^2(t)=u^3(t)$ and so $b=1$ and $\lambda_1=\lambda_2$ contradicting the original assumption. If $\beta_2<\beta_3$, according to Lemma~\ref{lemaA}, then $w^2_1(t)<w_1^3(t)$ and $u_1^2(t)>u_1^3(t)$ for all $t>0$. But this is impossible too, as then the first zeros of $u^3_1(t)$ and $w^3_1(t)$ cannot coincide, contradicting the fact that $u^3(t)$ is a dilation of $u^1(t)$. Finally, if $\beta_2>\beta_3$ we reach the same conclusion by analogous arguments. This completes the proof of uniqueness for the eigenvalue when $p=q$, and also the full proof of Theorem~\ref{existence}.
\end{proof}

The rest of this section is devoted to describing the symmetries of the eigenfunctions. The next statement is a direct consequence of the uniqueness part of Theorem~\ref{existence}.

\begin{Theorem} \label{Symmetricity Theorem}
Let $u(t)$ be a positive solution of \eqref{original} on $(0,t_0)$. Then, $u(t)=u(t_0-t)$ for all $0< t <\frac{t_0}{2}$. Moreover, $u(t)$ can be extended to a $2t_0$-periodic function $u_*\in C^1(\mathbb{R})$ satisfying \[ (\spower{u_*''(t)}{p})''=\lambda \spower{u_*(t)}{q}  \]
for all $t\in \mathbb{R}$.
\end{Theorem}
\begin{proof}
Start with a solution $u(t)$. Since the equation \eqref{original} is invariant under translations and the change $t\longmapsto -t$, then
$
    \tilde{u}(t)=u(t_0-t)
$ 
is also a solution of \eqref{original}. That is, for the same $\lambda>0$ and boundary conditions. But we have uniqueness in Theorem~\ref{existence}. For $p\not=q$, this implies directly that $u(t)=\tilde{u}(t)$. For $p=q$, it implies that $u(t)=c \tilde{u}(t)$. Then
\[
     u\left(\frac{t_0}{2}\right)=c\tilde{u}\left(\frac{t_0}{2}\right)=cu\left(\frac{t_0}{2}\right),
\]
so $c=1$ and once again $u(t)=\tilde{u}(t)$.

In order to achieve the second conclusion, note that $u\in C^1(0,t_0)$ and that the first conclusion implies that $u'(0)=-u'(t_0)$.
\end{proof}

\begin{Corollary} \label{Symmetricity Coll}
Let $u(t)$ be a solution of \eqref{original} with exactly $n-1$ zeros in $(0,t_0)$. Then, these zeros are all simple and located at
\[
    t_j=\frac{jt_0}{n} \qquad j=1,\ldots, n-1.
\]
Moreover, $u''(t)$ also vanish exactly at the points $t_j$.
\end{Corollary}
\begin{proof}
Let $u^1(t)$ be a positive solution of \eqref{original} on the segment $\left(0,\frac{t_0}{n}\right)$ for the same eigenvalue $\lambda>0$. From  Theorem~\ref{Symmetricity Theorem} it follows that $u^1_*(t)$ is a solution of the same equation on $(0,t_0)$ with exactly $n-1$ zeros in the interior of the segment. If $p=q$, as all eigenfunctions are multiple of one another, then the claimed conclusion follows. If $p\not=q$, by Corollary~\ref{Unique Theorem}, we obtain $u^1(t)=u(t)$ and the claimed conclusions follow. 
\end{proof}

We remark that, by combining the construction of Corollary~\ref{corollary1} with uniqueness, it is guaranteed that all the zeros of the eigenfunctions of \eqref{original} are simple.


\section{The case $q=p'$} \label{section5}

Let $\operatorname{B}(a,b)$ denote the Beta function and $\operatorname{B}_x(a,b)$ denote the incomplete Beta function \cite[8.17.1]{NIST}. 
For $1<r,s<\infty$, let
\begin{equation} \label{pi_p,q definition}
	\pi_{r,s}=2 \int_0^1\frac{\mathrm{d}\tau}{(1-\tau^s)^{\frac{1}{r}}} = \frac{2}{s} \operatorname{B}\!\left(\frac{1}{s},\frac{1}{r'}\right)
\end{equation} 
and let $\sin_{r,s}:\mathbb{R}\longrightarrow [-1,1]$ be the $2\pi_{r,s}$-periodic odd function whose inverse in $[0,\frac{\pi_{r,s}}{2}]$ is   
\[F_{r,s}(y)= \int_0^y \frac{\mathrm{d}\tau}{(1-\tau^s)^{\frac{1}{r}}} =\frac{1}{2}\operatorname{B}_{y^s}\!\!\left(\frac{1}{s},\frac{1}{r'}\right) 
\]
for all $y\in [0,1]$
and is even with respect to $\frac{\pi_{r,s}}{2}$.
Then, $\sin_{2,2}(x)=\sin(x)$ and, except possibly at the points $x=\frac{(2k+1)\pi_{r,s}}{2}$ for $k \in \Z$, the functions $\sin_{r,s}(x)$ are $C^{\infty}$. Note that \eqref{pi_p,q definition} and the fact that the Beta function is symmetric, yield the relation
	\begin{equation} \label{pi_pq dual formula}
		s \pi_{r,s}={r'} \pi_{s',r'}.
\end{equation}

In this section we show that, for $q=p'$, the systems ($*$) are solvable analytically in terms of $\sin_{r,s}$ and therefore the value of $\|E_2\|$ can be found explicitly in terms of $\operatorname{B}(a,b)$. 
The next two statements summarize our main finding.

\begin{Theorem} \label{E_2 q=p'}
	Let $1<p<\infty $ and fix $\mathcal{I}=[0,\pi_{2,p'}]$. Then
	\begin{equation} 
\|E_2\|=\sup_{f\in W_{\mathrm{D}}^{2,p}(\mathcal{I})} {\|f\|_{L^{p'}}  \over \|f'' \|_{L^p} }  = \left(\frac{2}{p'}\right)^{\frac{2}{p'}}\left( \operatorname{B}\!\left(\frac{1}{2}, \frac{p'+1}{p'} \right)\right)^{\frac{1}{p'}-\frac{1}{p}} \label{Th 5.1 eq}
	\end{equation} 
and the extremal functions are of the form $f(t)=c \sin_{2,p'}(t)$ where $c\in\mathbb{R}$ is a non-zero constant.
\end{Theorem}
 
Let us now re-write the equation \eqref{original} with the substitution $q=p'$. For $1<p<\infty$, we seek  for $u\not=0$ and $\lambda>0$ such that
		\begin{equation} \label{originalq=p'}
		\begin{aligned}
			&(\spower{u''}{p})''=\lambda \spower{u}{p'} & 0\leq t \leq t_0 \\
			&u(0)=u(t_0)=u''(0)=u''(t_0)=0.
		\end{aligned}
	\end{equation}

\begin{Theorem} \label{Theorem5}
The eigenvalues and eigenfunctions of \eqref{originalq=p'} are fully characterized as follows. For any given constant $c>0$ and $n\in\mathbb{N}$, $\lambda=\lambda_n$ is of the form
\[
	\lambda_n=\frac{\left( \pi_{2,p'} \pi_{p,2}n^2 \right)^p}{t_0^{2p}} c^{p-p'}= \left(\frac{p' \pi_{2,p'}^2 n^2}{2t_0^{2}}\right)^p c^{p-p'}
\]
with corresponding $u(t)=f_{n,c}(t)$ of the form
	\begin{equation} \label{n-eigenfunction}
		f_{n,c}(t)=c\sin_{2,p'}\left(\frac{\pi_{2,p'}n t}{t_0}\right).
	\end{equation}
For $p\not=2$, this eigenpair is the unique solution such that the eigenfunction has positive derivative at $x=0$ and changes sign exactly $n-1$ times on $(0,t_0)$. 
\end{Theorem}

We prove the validity of these two statements below. Let us begin by recalling properties of the generalized trigonometric functions and their role in the solution of the equation \eqref{classical_embedding}, associated to first order Sobolev embeddings $E_1$. The eigenpairs $(u,\lambda)$ of \eqref{classical_embedding}, have a close expression in terms of $\pi_{p,q}$ and $\sin_{p,q}$. Concretely, the full set of solutions of the Dirichlet problem \eqref{classical_embedding} is \cite{DM} 
\[u(t)\equiv u_{n,\alpha}(t)= {\alpha t_0 \over n\pi_{p,q}} \sin_{p,q}\left({n\pi_{p,q} \over t_0} \ t\right),
\]
for corresponding
\[ \lambda\equiv \lambda_{n,\alpha}= \left(\frac{n\pi_{pq}}{t_0}\right)^q \frac{|\alpha|^{p-q} q (1-p)}{p},
\]
where $\alpha\not=0$ is a real parameter and $n\in\mathbb{N}$. Here we focus on the case $q=p'$, which is independent of $\alpha$ only for $p=q=2$.

The generalised cosine, $\cos_{p,q}: \R \to [-1,1]$, is defined as 
\[\cos_{p,q}(x) = \frac{\mathrm{d}}{\mathrm{d} x} \sin_{p,q} (x), \qquad x \in \R.\]
From the properties of $\sin_{p,q}(x)$ it follows that $\cos_{p,q}(x)$ is an even, $2\pi_{p,q}$ periodic function, decreasing on $[0, \pi_{p,q}/2]$. If $x\in [0, \pi_{p,q}/2]$, then
\begin{equation} \label{cos_p,q formula}
	\cos_{p,q}(x)=(1-(\sin_{p,q} x)^{q})^{1/p} 
\end{equation}
and 
\[|\sin_{p,q}x|^q + |\cos_{p,q}x|^p=1\] 
for all $x\in \R.$

The functions $\sin_{p,p}$ and $\cos_{p,p}$ have a long history that can be traced back to about 40 years ago. Indeed, these and other analogue functions were examined by 
Schmidt \cite{Sch1}, Lindqvist \cite{Lin01,Lin02}, Elbert \cite{Elb}, and {\^ O}tani \cite{Ota01},  in connection with extremal functions for Hardy operators and eigenvalues of $p$-Laplacians. In recent years many properties of $\sin_{p,q}$  and $\cos_{p,q}$ have been discovered, mainly as a consequence of the study of their approximation characteristics. An account of this can be found in \cite{BM1,BM2,Melkonian2018}. It has also been discovered that these play a significant role in describing optimal Sobolev embeddings and related integral operators \cite{EL-book1}. 

The proof of Theorem~\ref{E_2 q=p'} relies on the decomposition $E_2=I_1 I_2$ where the embeddings $I_j$ are to be understood in the context of the following diagram,
 \begin{equation}
 	\begin{tikzcd}[column sep=small]
 		X=W^{2,p}_D \arrow{dr}{I_1} \arrow{rr}{E_2} & & Y=L^{p'} \\
 		&  Z=W^{1,2}_0 \arrow{ur}{I_2} & 
 	\end{tikzcd}
 \end{equation}
We will see that $u(t)=\sin_{2,p'}(\pi_{2,p'} t/t_0)$ is the extremal function for embedding  $I_2$ and $u'(t)$ is the extremal functions for $I_1$.  Hence, $u(t)$ is extremal for $E_2$ and eventually that would lead to $\|E_2\|=\|I_1\| \|I_2\|$. 

Before establishing the proofs of theorems~\ref{E_2 q=p'} and \ref{Theorem5}, we recall four additional known formulas connecting properties of the generalized trigonometric functions and their derivatives. See \cite[Lemma 2.2, Props. 3.1 and 3.2]{EGL}. Let $r,s \in (1,\infty)$. A direct calculation gives 
	\begin{align*}
	\cos_{r,s}' x&= -{r\over s}(\cos_{r,s} x)^{2-r}(\sin_{r,s} x)^{s-1} \qquad \text{and}\\ 
	 [(\sin_{r,s}x)^{r-1}]'&=(r-1)(\sin_{r,s}x)^{r-2}\cos_{r,s}x. \end{align*}
Thus,  as in \cite{EGL}, we get the general formula
	\begin{equation}
		\left[ \cos_{r,s}(\pi_{r,s} t /2)\right]^r=\left[\sin_{s',r'}(\pi_{s',r'}(1-t)/2)\right]^{r'}, 
		\label{sum formula p.q}
	\end{equation}
which we will employ mostly in the case $r=2$ and $s=p'$. Finally, let $u(x)=\sin_{2,p'}(x)$.
Applying \eqref{sum formula p.q}, then \eqref{pi_pq dual formula}, yields
		\begin{equation} 
	\label{u'' = [u(x)]...}
	u''(x)=-\spower{u(x)}{p'} {\pi_{p,2} \over \pi_{2,p'}}=-\spower{u(x)}{p'} {p' \over 2}.
\end{equation}
 
 \begin{proof}[Proof of Theorem \ref{E_2 q=p'}]
As $W_{\mathrm{D}}^{2,p}=W_0^{1,p}\cap W^{2,p}$, it is readily seen that \[W_{\mathrm{D}}^{2,p}=\left\{f\in W^{2,p}(\mathcal{I})\,:\, f(0)=0 \mbox{ and } \int_\mathcal{I} f'=0 \right\}.\] Moreover, $W_{\mathrm{D}}^{2,p} \subset W^{1,q}_0$ for any $1<q<\infty$ and in particular for $q=2$. We will use these facts below.

We begin by finding an upper bound for $\|E_2\|$. Note that
\[ 	\begin{aligned}
 		\|E_2\|&=\sup_{f\in W_{\mathrm{D}}^{2,p}}  {\|f\|_{L^{p'}}  \over \|f'' \|_{L^p} } \\ & = \sup_{f\in W_{\mathrm{D}}^{2,p}} {\|f\|_{L^{p'}}  \over \|f' \|_{L^2} } \,{\|f'\|_{L^2}  \over \|f'' \|_{L^p} } 	\\
 		& \le  \sup_{f\in W_{\mathrm{D}}^{2,p}} {\|f\|_{L^{p'}}  \over \|f' \|_{L^2} } \sup_{f\in W_{\mathrm{D}}^{2,p}} {\|f'\|_{L^2}  \over \|f'' \|_{L^p} }= N_1 N_2.
 	\end{aligned} 
 	\]
Let us compute the exact values of these constants.

On the one hand, we claim that 
\begin{equation} \label{eqN1} 
N_1=\frac{\pi_{2,p'}^{\frac12+\frac{1}{p'}} (2+p')^{\frac{1}{2}-\frac{1}{p'}} (p')^{\frac12}}{2^{\frac{1}{p}} \operatorname{B}(\frac{1}{p'},\frac{1}{2})}=(2+p')^{\frac{1}{2}-\frac{1}{p'}}
\operatorname{B}\!\left(\frac{1}{2},\frac{1}{p'}\right)^{\frac{1}{p'}-\frac12}
2^{\frac{2}{p'}-\frac{1}{2}}
(p')^{-\frac{1}{p'}}.
\end{equation}
The second equality follows from  \eqref{pi_p,q definition}.
To show the first equality, recall the following classical result of Talenti \cite[page 357]{Ta1} (see also \cite[(45.4)]{MPF}) for general segment $\mathcal{I}=[0,t_0]$. For all $1<r,s<\infty$,
 \begin{equation} \label{cubsopti}
		\sup_{f\in W^{1,r}_0} \frac{\|f\|_{L^s}}{\|f'\|_{L^r}}	 =\frac{t_0^{\frac{1}{r'}+\frac{1}{s}}(r'+s)^{\frac{1}{r}-\frac{1}{s}}(r')^{\frac{1}{s}} s^{\frac{1}{r'}}}{2 \operatorname{B}(\frac{1}{s},\frac{1}{r'})}
	\end{equation}
and the extremals of this are any non-zero multiple of $\sin_{r,s}({\pi_{r,s}t \over t_0})$.  Substituting $t_0=\pi_{2,p'}$, $r=2$ and $s=p'$, it follows that the middle expression of \eqref{eqN1} and the right hand side of \eqref{cubsopti} coincide. Then, since the optimizer $\sin_{2,p'}\in W_{\mathrm{D}}^{2,p}$, we have that
\[
      		N_1=\frac{\|\sin_{2,p'}\|_{L^{p'}}}{\|\cos_{2,p'}\|_{L^2}}\leq \sup_{f\in W^{2,p}_{\mathrm{D}}} \frac{\|f\|_{L^{p'}}}{\|f'\|_{L^2}}\leq \sup_{f\in W^{1,2}_0} \frac{\|f\|_{L^{p'}}}{\|f'\|_{L^2}} =N_1.
\] 
This confirms the first equality of \eqref{eqN1}.

On the other hand, we find the value of $N_2$ as follows. By substituting $f'=u$,
\[
     N_2=\sup_{f\in W^{2,p}_{\mathrm{D}}} \frac{\|f'\|_{L^2}}{\|f''\|_{L^p}}=\sup_{u\in S} \frac{\|u\|_{L^2}}{\|u'\|_{L^p}}
\]
where $S\subseteq \mathrm{Z}=\{u\in\mathrm{AC}(\mathcal{I})\,:\,\int_{\mathcal{I}} u=0\}$ is the subspace
\[
   S=\{u\in\mathrm{AC}(\mathcal{I})\,:\, u=f'\text{ for some }f\in W^{2,p}_{\mathrm{D}}(\mathcal{I})  \}.
\]
Now, let us show that
\begin{equation} \label{meanzero}
\sup_{u\in \mathrm{Z}} \frac{\|u\|_{L^2}}{\|u'\|_{L^p}}= 
		\frac{t_0^{\frac{1}{p'}+\frac12}(p'+2)^{\frac{1}{p}-\frac12}(p')^{\frac12} }{2^{\frac{1}{p}} \operatorname{B}(\frac12,\frac{1}{p'})}
\end{equation}
where the extremals are non-zero multiples of $\cos_{p,2}({\pi_{p,2}t \over t_0})$ on $\mathcal{I}=[0,t_0]$. 
Indeed, recall that the optimizer of the supremum on the left hand side is odd with respect to $t_0/2$ (see \cite{GN} or \cite{DGS}). Then,
\[
    \sup_{u\in \mathrm{Z}} \frac{\|u\|_{L^2}}{\|u'\|_{L^p}}=\sup_{v\in L^p(0,t_0/2)}\frac{\|Hv\|_{L^2(0,t_0/2)}}{\|v\|_{L^p(0,t_0/2)}},
\]
where $H:L^p(0,t_0/2)\longrightarrow L^2(0,t_0/2)$ is the Hardy operator.
According to \cite[Theorem 4.6]{EL-book1}, the supremum on the right hand side is attained whenever $u$ is any non-zero multiple of $\cos_{p,2}({\pi_{p,2}t \over t_0})$. Hence, from \eqref{cubsopti} with $r=p$ and $s=2$ together with the fact that
\[
   H\left[\cos_{p,2}(\pi_{p,2}(\cdot)/t_0)\right](t)=\frac{t_0\sin_{p,2}({\pi_{p,2}t \over t_0})}{\pi_{p,2}},
\] we obtain \eqref{meanzero} with the extremals as claimed. 
Now, for $t_0=\pi_{2,p'}$, $\cos_{p,2}(\frac{\pi_{p,2}}{\pi_{2,p'}}(\cdot))\in S$. Hence, 	
\[
\frac{\|\sin_{p,2}(\frac{\pi_{p,2}}{\pi_{2,p'}}(\cdot))\|_{L^2}}{\|\cos_{p,2}(\frac{\pi_{p,2}}{\pi_{2,p'}}(\cdot))\|_{L^p}}\leq 
N_2\leq \sup_{u\in {\mathrm{Z}}} \frac{\|u\|_{L^2}}{\|u'\|_{L^p}}=\frac{\|\sin_{p,2}(\frac{\pi_{p,2}}{\pi_{2,p'}}(\cdot))\|_{L^2}}{\|\cos_{p,2}(\frac{\pi_{p,2}}{\pi_{2,p'}}(\cdot))\|_{L^p}}.
\]
Thus
\[
    N_2=\frac{\pi_{p,2}^{\frac{1}{p'}+\frac12}(p'+2)^{\frac{1}{p}-\frac12}(p')^{\frac12}}{2^{\frac{1}{p}} \operatorname{B}(\frac12,\frac{1}{p'})}=(2+p')^{\frac{1}{p}-\frac12} 
\operatorname{B}\!\left(\frac12,\frac{1}{p'}\right)^{\frac{1}{p'}-\frac{1}{2}}
2^{\frac{2}{p'}-\frac{1}{2}}
 (p')^{-\frac{1}{p'}}
.
\]

Therefore, we get
 	\begin{equation} \label{last =}
 		N_1N_2=\left(\frac{2}{p'}\right)^{\frac{2}{p'}}\operatorname{B}\!\left(\frac12,\frac{p'+1}{p'}\right)^{\frac{1}{p'}-\frac{1}{p}}.
 	\end{equation}
So we have an upper bound for $\|E_2\|$. Let us now show that there is equality. Set  $u_1(x)= \sin_{2,p'}(x)$. By using \eqref{u'' = [u(x)]...}  we obtain
 	\[
 	 {\|u_1\|_{L^{p'}}  \over \|u_1'' \|_{L^p} }= {2 \over p'} \| u_1\|_{L^{p'}}^{1-\frac{p'}{p}}.
 	\]
 	 Now,
 	\[\int_0^{\pi_{2,p'}/2} \left( \sin_{2,p'} (x) \right)^{p'} \mathrm{d}x = \int_0^1 \frac{\tau^{p'}}{\left(1-\tau^{p'}  \right)^{1/2}} \mathrm{d}\tau=
 	\frac{1}{p'}\operatorname{B}\left(\frac{1}{2}, \frac{p'+1}{p'} \right),
 	\]
 	which can be obtain by substituting $\tau=\sin_{2,p'}(x),\, \mathrm{d}\tau=\cos_{2,p'}(x)$ and invoking identity \eqref{cos_p,q formula}. Thus,
 	\[  {\|u_1\|_{L^{p'}}  \over \|u_1'' \|_{L^p} }= \left(\frac{2}{p'}\right)^{1+\frac{1}{p'}(1-\frac{p'}{p})}\left( \operatorname{B}\left(\frac{1}{2}, \frac{p'+1}{p'} \right)\right)^{1/p'-1/p}   \]
 	which gives exactly the expression for $\|E_2\|$. 
 	
The uniqueness of the extremal function follows from the uniqueness of the extremal functions in the above arguments.
 \end{proof}

\begin{Remark} Evidently, Theorem~\ref{E_2 q=p'} follows from Theorem~\ref{Theorem5}, as the extremal function of $E_2$ is the first eigenfunction of \eqref{originalq=p'}. However, the proof we include above
has the advantage of clearly distinguishing the connection between first order embeddings and second order embeddings in the general case for $p$ and $q$. It shows that only for the case $q=p'$ the extremal functions coincide. 
\end{Remark}

  \begin{proof}[Proof of Theorem \ref{Theorem5}]
Let $t_0=\pi_{2,p'}$.
From \eqref{u'' = [u(x)]...}, applied twice,  we obtain that for $u(t)=\sin_{2,p'}(t)$ we have
\[((u''(t))^{p-1})''= -\left(\frac{p'}{2}\right)^{p-1}u''(t) =\left(\frac{p'}{2}\right)^{p}u^{p'-1}(t), 
\]
for all $0<t<\pi_{2,p'}$. As $u''(t)=-\frac{p'}{2}u^{p'-1}(t)$, then also $u''(0)=u''(\pi_{2,p'})=0$. Hence, $\sin_{2,p'}(t)$ is a positive eigenfunction for the problem  \eqref{originalq=p'} on $\mathcal{I}=[0,\pi_{2,p'}]$ with $\lambda = (p' / 2)^p$. 

For the general case $t_0>0$, observe that
\begin{equation} 
	f_{n,c}(t)=c\sin_{2,p}(\pi_{2,p'}n t/t_0),
\end{equation}
satisfies   \eqref{originalq=p'} on $[0,t_0]$ with eigenvalue
\[
\lambda_n=\frac{\left( \pi_{2,p'} \pi_{p,2}n^2 \right)^p}{t_0^{2p}} c^{p-p'}. \] 
Note that $\lambda_n$ is the $n$-th eigenvalue in the spectrum, for fixed $c>0$ and that $f_{n,c}$ has exactly $n-1$ zeros in $(0,t_0)$. Finally, uniqueness follows directly from Corollary~\ref{Unique Theorem}.
\end{proof}


\section{Approximation of Sobolev embedding} \label{section6}

In this final section we derive a precise connection between the s-numbers of $E_2$ and the eigenpairs of \eqref{original}. For this purpose, it is necessary to fix the norm of the eigenfunctions. We will call an eigenpair $({f}, {\lambda})$ a spectral couple of \eqref{original}, if $\|{f}''\|_{L^p}=1$ and $f'(0)>0$. Below we refer to such ${f}$ as a spectral function and we refer to the corresponding eigenvalue ${\lambda}>0$ as a spectral number. 

In earlier publications, the choice $\|\hat{f}\|_{L^q}=1$ is used. To distinguish the connection with our choice of normalization, we write $(\hat{f}, \hat{\lambda})$ for the corresponding spectral couple of this second kind. 

The connection between any eigenpair and either of the above choices becomes evident via re-scaling. Indeed, let $(\tilde{f},\tilde{\lambda})$ be any eigenpair of \eqref{original}. Then,  
\[\tilde{\lambda}=\frac{\|\tilde{f}''\|_{L^p}^p}{\|\tilde{f}\|_{L^q}^q}.
\]
Let 
\[ \alpha= \frac{\operatorname{sgn}(f'(0))}{\|\tilde{f}''\|_{L^p}}  \qquad \text{ and } \qquad \hat{\alpha}=\frac{1}{\|\tilde{f}\|_{L^q}}. \] Then, 
\[ \left(f(t)={\alpha} \tilde{f}(t),\quad \lambda= \left( \frac{\|\tilde{f}''\|_{L^p}}{\|\tilde{f}\|_{L^q}} \right)^q\right)\] is a spectral couple and 
\[ \left(\hat{f}(t)= \hat{\alpha} \tilde{f}(t), \quad   \hat{\lambda}= \left( \frac{\|\tilde{f}''\|_{L^p}}{\|\tilde{f}\|_{L^q}} \right)^p \right) \]
is a spectral couple of the second kind. Our discussion below only refers to spectral couples of the first kind.

The next lemma shows that the spectral functions form a unique chain, linked by re-scaling and generating a corresponding chain of spectral numbers. In turn, the latter form an increasing sequence of positive numbers accumulating at $+\infty$.

\begin{Lemma} \label{only one in spectrum}
Let $1<p,q<\infty$ and $n\in \mathbb{N}$. 
\begin{enumerate}[i)]
\item There is a unique spectral couple $(f,\lambda)$ such that $f$ has $n-1$ distinct zeros in $\operatorname{Int}(\mathcal{I})$.
\item  Let $(f_1,\lambda_1)$ be the spectral couple on $\mathcal{I}=[0,1]$ where $\lambda_1$ is the first spectral number. Let $f_{1*}:\mathbb{R}\longrightarrow \mathbb{R}$ be the 2-periodic odd function, such that $f_{1*}(t)=f_1(t)$ for all $t\in [0,1]$.  Then, 
	\[ \left(f_n(t)=\frac{f_{1*}(nt)}{n^2}, \quad \lambda_n=n^{2q}\lambda_1\right) \] 
	 is the spectral couple on $[0,1]$ associated to the $n$-th spectral number.
\item\label{Sn_not}	Let $(f_n, \lambda_n)$ be the $n$-th spectral couple of the previous item. Then, 
\[
    \left(\SN _n(t)=t_0^{2-1/p} f_n(t/t_0),\quad \sn _n=t_0^{q/p-1-2q} \lambda_n\right)
\]
is the spectral couple on $\mathcal{I}=[0,t_0]$ which has $n-1$ distinct zeros in $(0,t_0)$.
In particular
\[   \SN_n(t)=\frac{t_0^{2-1/p}}{n^2}f_{1*}(nt/t_0),\qquad \text{and} \qquad 
\sn _n =n^{2q}t_0^{q/p-1-2q}\lambda_1\]  
are the $n$-th spectral function and the $n$-th spectral number on $\mathcal{I}$, respectively.
\end{enumerate}
\end{Lemma}
\begin{proof}
 The first item follows directly from Theorem~\ref{Symmetricity Theorem} and Corollary~\ref{Symmetricity Coll}. The other statements follow from applying the substitutions \eqref{subst (1)}, then conducting the corresponding computations. 
\end{proof}

Below we will adhere dissambiguously to the notation of this lemma. The following direct consequence, is one of the main contributions of this paper. It gives an expression for the norm of the second order Sobolev embedding in terms of the eigenvalue equation \eqref{original}. 

\begin{Theorem} \label{Norma and eigenvalue}
For all $1<p,q<\infty$, the second order embedding \[E_2:W^{2,p}_{\mathrm{D}}(\mathcal{I}) \longrightarrow L^q(\mathcal{I})\] has norm
	\[\|E_2\|=\sup_{u\in W^{2,p}_{\mathrm{D}}} {\|u\|_{L^q} \over \|u''\|_{L^p}} = \|\SN_1\|_{L^q}=\sn _1^{-1/q}=|\mathcal{I}|^{1/q-1/p+2}\lambda_1^{-1/q},\]
	where $\lambda_1$ is the first spectral number on the unit interval $[0,1]$.
\end{Theorem}

It is natural now to describe the connection between the different s-numbers of $E_2$ and spectral couples. We begin by recalling the classical definitions.

\begin{Definition}
	Let  $s:T \mapsto \{s_n(T)\}\in \ell_\infty(\mathbb{N})$ be a rule which assigns to every bounded linear operator $T\in \operatorname{B}(X,Y)$ on every pair of Banach spaces $X$ and $Y$, a sequence of non-negative numbers satisfying the following properties.
	\begin{itemize}
		\item[(S1)] $\|T\|=s_1(T) \ge s_2(T) \ge ... 0$.
		\item[(S2)] $s_n(S+T) \le s_n(S)+\|T\| $ for $S,T \in \operatorname{B}(X,Y)$ and $n\in \N$.
		\item[(S3)] $s_n(BTA) \le \|B\| s_n(T)\|A\|$ whenever $A\in \operatorname{B}(X_0,X),$ $T\in \operatorname{B}(X,Y),$ $B\in \operatorname{B}(X,Y_0),$ and $n \in \N$.
		\item[(S4)] $s_n(Id:\mathbb{R}^n \to \mathbb{R}^n)=1$ for $n \in \N$.
		\item[(S5)] $s_n(T)=0$ when $\rank (T) <n$.
	\end{itemize}
We call $s_n(T)$ (or $s_n(T:X \to Y$)) an \emph{$n$-th s-number} of $T$. Moreover, when (S4) is replaced by
	\begin{itemize}
	\item[(S6)] $s_n(Id:E \to E)=1$ for every Banach space $E$ with $\dim(E)\ge n$,
		\end{itemize}
	we say that $s_n(T)$ is the $n$-th s-number of $T$ in the \emph{strict sense}. 
\end{Definition}

Many standard s-numbers of approximation theory are defined in relation to the moduli of injectivity and surjectivity, which we will recall.

\begin{Definition}
	Let $T\in \operatorname{B}(X,Y)$. The \emph{modulus of injectivity} of $T$ is
	\[j(T)= \sup\{\rho\ge 0: \|Tx\|_Y \ge \rho \|x\|_X \mbox{ for all } x\in X\}.\]
	The \emph{modulus of surjectivity} of $T$ is
	\[ q(T)= \sup\{\rho\ge 0: T(B_X) \supset \rho B_Y  \}.\]
\end{Definition}
Below we denote  the embedding of a closed linear subspace $M\subset X$ into $X$ by $J_M^X$ and the canonical map of $X$ onto the quotient space $X/ M$ by $Q_M^X$. The standard $n$-th s-numbers and their terminology is as follows.

\begin{Definition}
	Let $T\in \operatorname{B}(X,Y)$ and $n\in \N$. 
\begin{itemize}
\item The \emph{Approximation numbers} of $T$ are
	\[a_n(T)=\inf \{\|T-F\|: F\in B(X,Y),\,  \rank(F)< n \}.
	\]
\item The \emph{Isomorphism  numbers} of $T$ are
	\[i_n(T)=\sup \{\|A\|^{-1} \|B\|^{-1}\},\]
	the supremum taken over all  possible Banach spaces $G$ with $\dim(G) \ge n$ and maps $A\in B(Y,G),$ $B\in B(G,X)$ such that $ATB$ is the identity on $G$.
\item The \emph{Gelfand  numbers} of $T$ are
\[c_n(T) = \inf\{\|TJ_M^X\|: \codim(M) < n\}.\]
\item The \emph{Bernstein  numbers} of $T$ are
\[b_n(T)=\sup\{j(TJ_M^X): \dim(M)\ge n  \}.  \]
\item The \emph{Kolmogorov numbers} of $T$ are
\[d_n(T)=\inf \{\|Q_N^YT\|: \dim(N) < n \} .\]
\item The \emph{Mityagin  numbers} of $T$ are
\[m_n(T)= \sup\{q(Q_N^YT): \codim(N)\ge n \}.\]
\end{itemize}
\end{Definition}

In the context of this definition, recall the fundamental relation \cite{EL-book1,Pie01}
\begin{equation} \label{s-numbers relations}
	\begin{aligned} 
	a_n(T) \ge &  \max[c_n(T),d_n(T)] \ge \min[c_n(T), d_n(T)]  \\
	 \ge & \max[b_n(T),m_n(T)] \ge \min[b_n(T),m_n(T)] \ge i_n(T).
\end{aligned}
\end{equation}
Moreover, the approximation numbers are the largest s-numbers and the isomorphism numbers are the smallest strict s-numbers.

\begin{Theorem} \label{Th from BeTi2}
	For all $1<p,q<\infty$, the second order embedding \[E_2:W^{2,p}_{\mathrm{D}}(\mathcal{I}) \longrightarrow L^q(\mathcal{I})\] has s-numbers obeying the following relations. 
\begin{enumerate}[i)] \item \label{s-num1} If $p<q$, then
	\[i_n(E_2)\geq \sn_n^{-1/q}={|\mathcal{I}|^{1/q+2-1/p}\over n^2 \lambda_1^{1/q}}.
	\]
\item \label{s-num2} If $p \ge q$, then
	\[a_n(E_2)\leq \sn_n^{-1/q}= {|\mathcal{I}|^{1/q+2-1/p}\over n^2 \lambda_1^{1/q}}.
	\]
\end{enumerate}
\end{Theorem}
\begin{proof}
Fix $1<p \le q < \infty$. 

\underline{Proof of i)}. We show that 
\begin{equation*}
i_n(E_2)\ge \sn_n^{-1/q}.
\end{equation*}

Set $0=a_0<a_1< ... < a_n=t_0$ where $a_i-a_{i-1}=t_0/n$ and $\mathcal{I}_i=(a_{i-1},a_i)$. Then, $\SN_n(a_i) = 0=\SN''_n(a_i)$ for $i=0, ... , n$. Let $g_i=\SN_n \chi_{\mathcal{I}_i}$ for $1\le i \le n$. Then  
$M_n=\operatorname{span}\{g_1, ... , g_n\}$ has dimension $n$. Also, $\|g''_i\|_{L^p}=\|g''_1\|_{L^p}$ for every $1\le i \le n$. 

Now, let
$\mu= \sn_1$ for $\mathcal{I}_1$. By virtue of Theorem~\ref{Norma and eigenvalue} applied on the interval $\mathcal{I}_i$,
\[
    \SN_1(t)=\frac{\operatorname{sgn}(g_i'(0))}{\|g_i''\|_{L^p}}g_i(t) 
\]
for all $t\in \mathcal{I}_i$.
Hence,
\[
     \|g_i\|_{L^q}^q=\mu^{-1}\|g_1''\|_{L^p}^q 
\]
for all $i\in\{1, \ldots, n\}$.

Recall the next identity valid for $p\leq q$, \cite[Lemma 8.14]{EL-book1},
\begin{equation} \label{lemma p<q}
	\inf_{\underline{\alpha} \in \R^n} \frac{\left( \sum_{i=1}^n |\alpha_i|^q \right)^{1/q}}
{\left( \sum_{i=1}^n |\alpha_i|^p \right)^{1/p}}=n^{1/q-1/p},
\end{equation}
and the infimum is attained when $|\alpha_i|=c$ , $i=1,...,n$.  
Since the supports of the $g_i$ are disjoint, we have,
\[\begin{aligned}
 \inf_{0\not=u\in M_n} \frac{\|u\|_{L^q}}{\|u''\|_{L^p}}
= & \inf_{\underline{\alpha}\in \R^n\setminus \{0\}} \frac{\|\sum_{i=1}^{n}\alpha_i g_i\|_{L^q}}{\|\sum_{i=1}^{n} \alpha_i g''_i\|_{L^p}} = \inf_{\underline{\alpha}\in \R^n\setminus \{0\}} \frac{(\sum_{i=1}^{n} \|\alpha_i g_i\|^q_{L^q})^{1/q}}{(\sum_{i=1}^{n} \| \alpha_i g''_i\|^p_{L^p})^{1/p}}\\
= & \inf_{\underline{\alpha}\in \R^n\setminus \{0\}} \frac{(\sum_{i=1}^{n} \|\alpha_i g''_1\|^q_{L^p} \mu^{-1})^{1/q}}{(\sum_{i=1}^{n} \| \alpha_i g''_1\|^p_{L^p})^{1/p}} = \inf_{\underline{\alpha}\in \R^n\setminus \{0\}} \frac{(\sum_{i=1}^{n} |\alpha_i|^q )^{1/q}}{(\sum_{i=1}^{n} | \alpha_i|^p)^{1/p}} \mu^{-1/q} \\ = & n^{1/q-1/p} \mu^{-1/q}= \sn_n^{-1/q}.
\end{aligned} 
\]
The last equality follows from Lemma~\ref{only one in spectrum}.

This allows us to complete the proof of i) as follows. In the definition of the isomorphism numbers, replace $G=M_n$ with norm $\|\cdot\|_{L^q}$, $X=W^{2,p}_{\operatorname{D}}$, $Y=L^q$ and $T=E_2$. Now set $B:G\longrightarrow X$ the operator $B(u)=u$ and $A:Y\longrightarrow G$ the projection obtained by completing $\{g_i\}$ to a basis of $L^q$. Then $\|A\|=1$ and 
\[
    \|B\|^{-1}=\left(\sup_{0\not=u\in M_n}\frac{\|u''\|_{L^p}}{\|u\|_{L^q}}\right)^{-1}= \inf_{0\not=u\in M_n} \frac{\|u\|_{L^q}}{\|u''\|_{L^p}} =\sn_n^{-1/q}.
\]
So,  indeed $i_n(E_2)\geq \sn_n^{-1/q}$.

\underline{Proof of ii)}. 
Now we claim that for $q \le p$, 
\begin{equation} \label{stp2} 
a_n(E_2)\le \sn_n^{-1/q}.
\end{equation}
Using the partition of the interval into $n$ sub-intervals as in the previous part, we now let
\[
(T_{i}u)(x)=
\chi_{\mathcal{I}_i}(x)\left(u(a_{i-1})+\frac{u(a_i)-u(a_{i-1})}{a_i-a_{i-1}}(x-a_{i-1})\right)
\]
and $T=\sum_{i=1}^nT_i$. Note that $T$ defined on $W^{2,p}_{\mathrm{D}}(\mathcal{I})$ is an operator of rank $n-1$. Then,
\[
    a_n(E_2)\leq \sup_{0\not=f\in W^{2,p}_{\mathrm{D}}} \frac{\|f-Tf\|_{L^q}}{\|f''\|_{L^p}}\leq \sup_{0\not=u\in W^{2,p}} \frac{\|u-Tu\|_{L^q}}{\|u''\|_{L^p}}.
\]
Now, for any $0\not=u\in W^{2,p}$,
\[
    \|u-Tu\|_{L^q}=\left( \sum_{i=1}^n \|u-T_i u\|^q_{L^q(\mathcal{I}_i)}\right)^{1/q}.
\]
Also, for all $x\in \mathcal{I}_i$, $u''(x)=(u-Tu)''(x)$. Hence,
\[
    \|u''\|_{L^p}=\left( \sum_{i=1}^n \|(u-T_i u)''\|^p_{L^p(\mathcal{I}_i)}\right)^{1/p}.
\]

Let $u_i=\chi_{\mathcal{I}_i} u$. Then, for any $u_i\in W^{2,p}(\mathcal{I}_i)$ we have $(u_i-T_iu_i)(a_i)=(u_i-T_iu_i)(a_{i-1})=0$ and $(u_i-T_iu_i)''=u_i''$. Thus,
\begin{align*}
   a_n(E_2)&\leq \sup_{u\in W^{2,p}} \frac{\left( \sum_{i=1}^n \|u_i-T_i u_i\|^q_{L^q}\right)^{1/q}}{\left( \sum_{i=1}^n \|u_i''\|^p_{L^p}\right)^{1/p}}\\
&\leq \sup_{\substack{u_i\in W^{2,p}(\mathcal{I}_i) \\ i=1,\ldots,n}} \frac{\left( \sum_{i=1}^n \|u_i-T_i u_i\|^q_{L^q}\right)^{1/q}}{\left( \sum_{i=1}^n \|u_i''\|^p_{L^p}\right)^{1/p}} \\
&= \sup_{\substack{f_i\in W^{2,p}_{\mathrm{D}}(\mathcal{I}_i) \\ i=1,\ldots,n}} \frac{\left( \sum_{i=1}^n \|f_i\|^q_{L^q}\right)^{1/q}}{\left( \sum_{i=1}^n \|f_i''\|^p_{L^p}\right)^{1/p}}.
\end{align*}

By virtue of Theorem~\ref{Norma and eigenvalue} applied on the sub-segments $\mathcal{I}_i$, which are of equal length, we then have
\begin{align*} 
\sup_{\substack{f_i\in W^{2,p}_{\mathrm{D}}(\mathcal{I}_i) \\ i=1,\ldots,n}} & \frac{\left( \sum_{i=1}^n \|f_i\|^q_{L^q}\right)^{1/q}}{\left( \sum_{i=1}^n \|f_i''\|^p_{L^p}\right)^{1/p}} \\ &\leq \lambda_1^{-1/q} |\mathcal{I}_1|^{2+1/q-1/p}\sup_{\substack{f_i\in W^{2,p}_{\mathrm{D}}(\mathcal{I}_i) \\ i=1,\ldots,n}} \frac{\left( \sum_{i=1}^n \|f_i''\|^q_{L^p}\right)^{1/q}}{\left( \sum_{i=1}^n \|f_i''\|^p_{L^p}\right)^{1/p}}\\ &={|\mathcal{I}|^{1/q+2-1/p}\over n^{2+1/q-1/p} \lambda_1^{1/q}}\sup_{\underline{\alpha} \in \R^n} \frac{\left( \sum_{i=1}^n |\alpha_i|^q \right)^{1/q}}
{\left( \sum_{i=1}^n |\alpha_i|^p \right)^{1/p}},
\end{align*}
where $|\alpha_i|=\|f_i''\|_{L^p}$.

Invoke now the identity \cite[Lemma 8.23]{EL-book1}, valid for $q\leq p$,
\begin{equation} \label{lemma p>q} \sup_{\underline{\alpha} \in \R^n} \frac{\left( \sum_{i=1}^n |\alpha_i|^q \right)^{1/q}}
{\left( \sum_{i=1}^n |\alpha_i|^p \right)^{1/p}}=n^{1/q-1/p},
\end{equation}
where the supremum being attained when $|\alpha_i|=c$ , $i=1,...,n$.
Hence, we finally get
\[
    a_n(E_2)\leq {|\mathcal{I}|^{1/q+2-1/p}\over n^{2} \lambda_1^{1/q}}=
     \sn_n^{-1/q}.
\]
as claimed.
\end{proof}

The second statement in the previous theorem sharpens the results obtained in \cite{B1} and \cite{BeTi}. 

The next statement encompasses the other original main purpose of this paper, but we omit details of its proof, which is beyond the current scope. It follows from Lemma~\ref{only one in spectrum} and Theorem~\ref{Th from BeTi2}, alongside with results settled in \cite{B1} and  \cite[``Basic Theorem'']{BeTi}. An interesting related question, is whether we can replace $d_n(E_2)$ by $c_n(E_2)$ and $b_n(E_2)$ by $m_n(E_2)$ below.  Our omission of the proof is motivated by the fact that we aim at reporting on a full investigation of this question in future.

\begin{Theorem} \label{Lemma from BeTi1}
Let $1<p,q<\infty$.
\begin{itemize}
\item[i)]	If $p\le q$, then
	\[i_n(E_2)=b_n(E_2)=\sn_n^{-1/q}={|\mathcal{I}|^{1/q+2-1/p}\over n^2 \lambda_1^{1/q}}.
	\]
\item[ii)]	If $q \le p$, then
	\[a_n(E_2)=d_n (E_2)=\sn_n^{-1/q}= {|\mathcal{I}|^{1/q+2-1/p}\over n^2 \lambda_1^{1/q}}.
	\]
\end{itemize}
\end{Theorem}


\section*{Acknowledgments} This research was funded by The UK's Royal Society International Exchange Grant ``James Orthogonality and Higher order Sobolev Embeddings''. We are also grateful to our colleagues at the Czech Technical University in Prague for hosting many of the discussions that eventually lead to this paper. 

\def\bstname{mn}

\end{document}